\documentclass[12pt]{amsart}

\usepackage{amsmath,amsfonts,amssymb}

\theoremstyle{plain}
\newtheorem{lemma}{Lemma}[section]
\newtheorem{theorem}[lemma]{Theorem}

\newtheorem{proposition}[lemma]{Proposition}

\theoremstyle{definition}
\newtheorem{definition}[lemma]{Definition}

\numberwithin{equation}{section}
\begin{document}

\newcommand{\ZZ}{\mathbb{Z}}
\newcommand{\ZZd}{\mathbb{Z}^{d}}
\newcommand{\RR}{\mathbb{R}}
\newcommand{\RRd}{\mathbb{R}^{d}}
\newcommand{\PP}{\mathbb{P}}
\newcommand{\QQ}{\mathbb{Q}}
\newcommand{\EE}{\mathbb{E}}
\newcommand{\mB}{\mathcal{B}}
\newcommand{\mC}{\mathcal{C}}
\newcommand{\mD}{\mathcal{D}}
\newcommand{\mE}{\mathcal{E}}
\newcommand{\mF}{\mathcal{F}}
\newcommand{\mG}{\mathcal{G}}
\newcommand{\mH}{\mathcal{H}}
\newcommand{\mI}{\mathcal{I}}
\newcommand{\mJ}{\mathcal{J}}
\newcommand{\mL}{\mathcal{L}}
\newcommand{\mM}{\mathcal{M}}
\newcommand{\mkN}{\mathfrak{N}}
\newcommand{\mO}{\mathcal{O}}
\newcommand{\mQ}{\mathcal{Q}}
\newcommand{\mR}{\mathcal{R}}
\newcommand{\mS}{\mathcal{S}}
\newcommand{\mT}{\mathcal{T}}
\newcommand{\mU}{\mathcal{U}}
\newcommand{\mW}{\mathcal{W}}
\newcommand{\mY}{\mathcal{Y}}
\newcommand{\bs}{\backslash}
\newcommand{\half}{\frac{1}{2}}
\newcommand{\fN}{\frac{1}{N}}
\newcommand{\bx}{\mathbf{x}}
\newcommand{\bV}{\mathbf{V}}
\newcommand{\olV}{\overline{V}}
\newcommand{\olm}{\overline{m}}
\newcommand{\olf}{\overline{f}}
\newcommand{\olS}{\overline{S}}
\newcommand{\hZ}{\hat{Z}}
\newcommand{\hJ}{\hat{\mathcal{J}}}
\newcommand{\ha}{\hat{a}}
\newcommand{\hb}{\hat{b}}
\newcommand{\hl}{\hat{l}}
\newcommand{\hmJ}{\hat{\mathcal{J}}}
\newcommand{\hmW}{\hat{\mathcal{W}}}
\newcommand{\hmS}{\hat{\mathcal{S}}}
\newcommand{\hmJs}{\hat{\mathcal{J}}^*}
\newcommand{\hmL}{\hat{\mathcal{L}}}
\newcommand{\tq}{\tilde{q}}
\newcommand{\tF}{\tilde{F}}

\newcommand{\tmJ}{\tilde{\mathcal{J}}}
\newcommand{\ta}{\tilde{\alpha}}
\newcommand{\tnu}{\tilde{\nu}}
\newcommand{\tphi}{\tilde{\varphi}}
\newcommand{\tm}{\tilde{m}}
\newcommand{\tx}{\tilde{x}}
\newcommand{\tbx}{\tilde{\mathbf{x}}}
\newcommand{\tdelta}{\tilde{\delta}}
\newcommand{\tE}{\tilde{E}}
\newcommand{\hN}{\hat{N}_0}
\newcommand{\ep}{\epsilon}
\newcommand{\hS}{\hat{S}}
\newcommand{\hd}{\hat{\delta}}
\newcommand{\hp}{\hat{\varphi}}
\newcommand{\tS}{\tilde{S}}

\title[Quenched vs. Annealed Critical Points]{Quenched and Annealed Critical Points in Polymer Pinning Models}
\author{Kenneth S. Alexander}
\address{Department of Mathematics KAP 108\\
University of Southern California\\
Los Angeles, CA  90089-2532 USA}
\email{alexandr@usc.edu}
\thanks{The research of the first author was supported by NSF grant DMS-0405915.}
\author{Nikos Zygouras}
\email{zygouras@usc.edu}


\begin{abstract}
We consider a polymer with configuration modeled by the path of a Markov chain, interacting with a potential $u+V_n$ which the chain encounters when it visits a special state 0 at time $n$.  The disorder $(V_n)$ is a fixed realization of an i.i.d. sequence.  The polymer is pinned, i.e. the chain spends a positive fraction of its time at state 0, when $u$ exceeds a critical value.  We assume that for the Markov chain in the absence of the potential, the probability of an excursion from 0 of length $n$ has the form $n^{-c}\varphi(n)$ with $c \geq 1$ and $\varphi$ slowly varying.  Comparing to the corresponding annealed system, in which the $V_n$ are effectively replaced by a constant, it was shown in \cite{Al08}, \cite{DGLT07}, \cite{To07} that the quenched and annealed critical points differ at all temperatures for $3/2<c<2$ and $c>2$, but only at low temperatures for $c<3/2$.  For high temperatures and $3/2<c<2$ we establish the exact order of the gap between critical points, as a function of temperature.  For the borderline case $c=3/2$ we show that the gap is positive provided $\varphi(n) \to 0$ as $n \to \infty$, and for $c >3/2$ with arbitrary temperature we provide an alternate proof of the result in \cite{DGLT07} that the gap is positive, and extend it to $c=2$.
\end{abstract}

\maketitle

\section{Introduction}
A polymer pinning model is described by a Markov chain $(X_n)_{n\geq 0}$ on a state state space $\Sigma$, containing a special point $0$ where the polymer interacts with a potential. The space-time trajectory of the Markov chain represents the physical configuration of the polymer, with the $n$th monomer of the polymer chain located at $(n,X_n)$, or alternatively, one can view $X_n$ as the location of the $n$th monomer, with $n$ being just an index; these are mathematically equivalent.  We denote the distribution of the Markov chain in the absence of the potential, started from 0, by $P^X$ and we assume that it is recurrent and has an excursion length distribution (from the $0$ state) with power-law decay:
\begin{eqnarray}\label{excursion_law}
P^X(\mE=n)=\frac{\varphi(n)}{n^c},\qquad n\geq 1.
\end{eqnarray} 
Here $\mE$ denotes the length of an excursion from $0$, $c \geq 1$, and $\varphi(\cdot)$ is a slowly varying function, that is, a function satisfying $\varphi(\kappa n)/\varphi(n)\to 1$ as $n$ tends to infinity, for all $\kappa>0$.  

When the chain visits 0 at some time $n$, it encounters a potential of form $u+V_n$, with the values $V_n$ typically modeling variation in monomer species.  This (quenched) pinning model is described by the Gibbs measure
\begin{eqnarray}\label{polymer_measure}
d\mu^{\beta,u,\bV}_{N}(\bx) = \frac{1}{Z_{N}}e^{\beta H_N^u(\bx,\bV)}\, dP^X(\bx)
\end{eqnarray}
where $\bx = (x_n)_{n \geq 0}$ is a path, $\bV = (V_n)_{n \geq 0}$ is a realization of the disorder, and 
\begin{equation} \label{Hdef}
  H_N^u(\bx,\bV) = \sum_{n=0}^N (u+V_n)\delta_0(x_n)
  \end{equation}
and the normalization 
\begin{eqnarray*}
Z_{N}= Z_N(\beta,u,\bV) = E^X\left[e^{\beta H_N^u(\bx,\bV)}\, \right]
\end{eqnarray*}
 is the partition function.  The disorder $\bV$ is a sequence
of i.i.d.~random variables with mean zero, variance one and finite exponential moments; we assume they are Gaussian here to keep the exposition simple, and we denote the distribution of this sequence by $P^V$. The parameter $u\in\mathbb{R}$ is thus the mean value of the potential, and $\beta>0$ is the inverse temperature. 

One would like to understand how the presence of the random potential affects the path properties of the Markov chain, and in particular how the case with disorder differs from the homogeneous case $V_n \equiv 0$.  These effects can be quantified via the free energy and the contact fraction.  To be more precise, letting $L_N = L_N(\bx) = \sum_{n=0}^N\delta_0(x_n)$ denote the local time at $0$, it is proved in \cite{AS06} that there exists a nonrandom $C_q(\beta,u)$ such that
\begin{eqnarray*}
  \lim_{N\to\infty} E_{\mu_{N}^{\beta,u,\bV}} \left(\,\frac{L_N}{N}\, \right)= C_q(\beta,u),\qquad P^V-a.s.
\end{eqnarray*}
for every $\ep>0$;
$C_q(\beta,u)$ is called the {\it quenched contact fraction}.
We will say that the polymer is {\it pinned} at $(\beta,u)$ if $C_q(\beta,u)>0$ and {\it depinned} if $C_q(\beta,u)=0$.  
Monotonicity in $u$ is clear so there exists $u_c^q(\beta)$ such that the polymer is pinned for $u>u_c^q(\beta)$ and depinned for $u < u_c^q(\beta)$. Note that when $c<2$ the Markov chain is null recurrent and the set of paths with any given positive contact fraction is exponentially rare, so pinning requires a compensating energy gain from the potential to offset this entropy cost.  Pinning can also be described in terms of the {\it quenched free energy} $f_q(\beta,u)$ given by
\begin{equation} \label{freeenergy}
  \beta f_q(\beta,u) = \lim_{N \to \infty} \frac{1}{N} \log Z_N(\beta,u,\bV);
  \end{equation}
the fact that the free energy exists and is nonrandom (off a null set of disorders) is proved in \cite{AS06}.  The free energy is 0 if $u < u_c^q(\beta)$ and strictly positive if $u > u_c^q(\beta)$.  The free energy and contact fraction are related by
\[
  C_q(\beta,u) = \frac{\partial}{\partial u} f_q(\beta,u).
  \]

The effect of disorder is studied by comparing the quenched pinning model to its annealed version, obtained by averaging the Gibbs weight over the disorder:
\begin{eqnarray}\label{annealed_measure}
d\nu_N^{\beta,u}=\frac{1}{E^VZ_{N}}e^{\beta(u + \beta^{-1}\log M_V(\beta))L_N}\, dP^X,
\end{eqnarray}
where $M_V(\beta) = E^V(e^{\beta V_1})$ is the moment generating function.  The annealed model is thus equivalent to a quenched model with $V_n \equiv 0$ and $u$ replaced by $u + \beta^{-1}\log M_V(\beta)$, and it is readily shown (see \cite{Gi06}) that the critical point $u_c^a(\beta)$ in the annealed model is the point where the exponent in \eqref{annealed_measure} is 0, which in the Gaussian case means $u_c(\beta) = -\beta/2$.  It is therefore natural to define the variable $\Delta$ by 
\[
  u = -\frac{\beta}{2} + \Delta,
  \]
giving critical points $\Delta_c^a(\beta) = 0$ and $\Delta_c^q(\beta)$.  We then have (cf. \eqref{annealed_measure})
\begin{equation} \label{EVEX}
  E^V Z_N(\beta,u,\bV) = E^X e^{\beta\Delta L_N}.
  \end{equation}
The annealed contact fraction and free energy are given by
\begin{equation} \label{annfreeen}
  \beta f_a(\beta,u) = \lim_{N \to \infty} \frac{1}{N} \log E^V Z_N(\beta,u,\bV) =  \lim_{N \to \infty} \frac{1}{N} \log 
    E^X e^{\beta\Delta L_N}
  \end{equation}
and
\[
  C_a(\beta,u) = \lim_{N\to\infty} E_{\nu_{N}^{\beta,u}} \left(\,\frac{L_N}{N}\,\right)
    = \frac{\partial}{\partial u} f_a(\beta,u)
  \]
respectively.  Since $E^V( \log  Z_N(\beta,u,\bV) ) \leq \log E^V(  Z_N(\beta,u,\bV) )$, we have $f_q \leq f_a$ and therefore $\Delta_c^q(\beta) \geq 0$.  

It is proved in \cite{Al08} that
\begin{equation} \label{Casymp}
  C_a(\beta,u) \sim (\beta\Delta)^{\frac{2-c}{c-1}}\hat{\varphi}_{c-1}\left( \frac{1}{\beta\Delta} \right)
    \quad \text{ as } \beta\Delta \searrow 0
  \end{equation}
for $c<2$, while for $c>2$ the transition is discontinuous:
\begin{equation} \label{Casymp2}
  C_a(\beta,u) \to \frac{1}{E^X(\mE)} > 0 \quad \text{ as } \beta\Delta \searrow 0.
  \end{equation}
Here $\hat{\varphi}_{c-1}$ is a slowly varying function related to $\varphi$; see the proof of Lemma \ref{concavity2} below. This means that the annealed specific heat exponent (which is, roughly speaking, the exponent $\alpha$ such that the free energy decreases as $\Delta^{2-\alpha}$ as $\Delta \to 0$) is $(2c-3)/(c-1)$.  A strong effect of disorder is evident when the specific heat exponent and/or critical point differ between quenched and annealed systems.  In the physics literature, the disorder is said to be {\it relevant} if these specific heat exponents differ.  Predictions from that literature were confirmed rigorously when it was shown that the disorder is relevant for $c>3/2$, i.e. when the specific heat exponent is positive \cite{GT06a}, and (for small $\beta$) irrelevant for $c<3/2$ \cite{Al08}.   In \cite{Al08} the quenched and annealed critical points were also proved equal ($\Delta_c^q(\beta) = 0$) for small $\beta$ when $c<3/2$, and very recently in \cite{DGLT07} it was proved that $\Delta_c^q(\beta)>0$  for all $\beta>0$ when $3/2<c<2$ and when $c>2$, as well as for large $\beta$ with arbitrary $c>1$.  Alternate proofs of these results from \cite{Al08} appear in \cite{To08}.

In \cite{Al08} the following was proved for $3/2<c<2$ and $\beta$ sufficiently small.  In contrast to \eqref{Casymp}, which has an infinite derivative at $\Delta = 0$, we have the linear bound
\[
  C_q(\beta,u) \leq \frac{2\Delta}{\beta},
  \]
so if we define $\Delta_0 = \Delta_0(\beta)$ by 
\[
  \frac{2\Delta}{\beta} = (\beta\Delta)^{\frac{2-c}{c-1}}\hat{\varphi}_{c-1}\left( \frac{1}{\beta\Delta} \right),
  \]
we see that $C_q(\beta,u)$ is forced to be smaller than $C_a(\beta,u)$ for (roughly) $\Delta < \Delta_0$, and in fact $C_q(\beta,u) = o(C_a(\beta,u))$ as $\Delta \to 0$.  On the other hand, given $\ep>0$ there is a $K = K(\ep)$ such that
\begin{equation} \label{largeDelta}
  \left| \frac{C_q(\beta,u)}{C_a(\beta,u)} - 1 \right| < \ep \quad \text{for all } \Delta > K\Delta_0.
  \end{equation}
Up to a constant, then, the value
\[
  \Delta_0(\beta) \sim K_1\beta^{1/(2c-3)} \hp_{c - \frac{3}{2}}\left( \frac{1}{\beta} \right)^{1/2}
  \]
separates those (small) values of $\Delta$ for which the disorder significantly reduces the contact fraction, from those (larger) values for which it does not.  Here $\hp_{c - \frac{3}{2}}$ is another slowly varying function related to $\varphi$.  It should be noted that our $\Delta_0(\beta)$ here is essentially the $\Delta_1(\beta)$ defined in \cite{Al08}, while $\Delta_0(\beta)$ in \cite{Al08} denotes a quantity which is asymptotically a constant multiple of our $\Delta_0(\beta)$ here.  Since we only care about the order of magnitude here, the difference is not significant.

The quenched and annealed polymers, then, must behave quite differently for $\Delta \ll \Delta_0$.  It is useful to describe this heuristically in terms of strategies, by which we mean classes of qualitatively similar paths.  For $\Delta>0$ the strategy of the annealed polymer is essentially to alter its excursion length distribution (compared to $P^X$) so that the mean becomes $1/C_a(\beta,u)$.  The altered distribution which minimizes the relative entropy has the form
\[
  \nu_\alpha(\mE=n) = \frac{e^{-\alpha n} P^X(\mE=n) }{E^X[e^{-\alpha\mE}]}, \quad n \geq 1,
  \]
with $\alpha$ chosen to give the desired mean $E_{\nu_{\alpha}}(\mE) = 1/C_a(\beta,u)$; this is achieved for $\alpha = \beta f_a(\beta,u)$ \cite{GT06a}.  This is the limiting distribution (as $N \to \infty$) for the excursion length in the annealed polymer \cite{Gi06}.  If the quenched polymer is pinned for $\Delta \ll \Delta_0$, it must employ a substantially different strategy, since the quenched contact fraction must be a small fraction of the annealed one.  An example of a candidate for such an alternate strategy is the {\it Imry-Ma strategy}, which essentially consists of locating those rare ``rich'' segments in which the average disorder value is exceptionally large, and making long excursions from one rich segment to another.  The Imry-Ma strategy has been studied in some related contexts (\cite{BG04},\cite{GT06a}).  The significant use, or not, of alternate strategies (Imry-Ma or otherwise) can thus be quantified, at least heuristically, by whether or not $\Delta_c^q(\beta)$ is $o(\Delta_0)$ as $\beta \to 0$.
Our main result here says that alternate strategies are not used significantly:  there exists $\ep_0$ such that the quenched polymer is not pinned when $\Delta < \ep_0\Delta_0$.  In combination with \eqref{largeDelta} this says (still heuristically) that the ability of the quenched polymer to mimic the annealed one breaks down entirely as $\Delta$ passes down through order $\Delta_0$. 

An alternate description of $\Delta_0$ is as follows.  Consider a block of monomers extending one annealed correlation length, that is, having length $M \approx (\beta f_a(\beta,u))^{-1}$.  The fluctuations in the average $\olV = \sum_{i=1}^M V_i$ of the disorder over such a block are of typical order $M^{-1/2}$.  If this typical fluctuation is at least of the order of $\Delta$, then blocks with average potential $u + \olV<u_c^a$ (that is, $\Delta + \olV < 0$) will be relatively common.  In such ``bad'' blocks it will typically not be energetically advantageous for the quenched polymer to be pinned.  It is easily shown using the asymptotics established in \cite{Al08} that for $\Delta = \Delta_0$, $M^{-1/2}$ and $\Delta$ are of the same order as $\beta \to 0$, while $M^{-1/2} \gg \Delta$ if $\Delta \ll \Delta_0$, and $M^{-1/2} \ll \Delta$ if $\Delta \gg \Delta_0$.  Thus as $\Delta \searrow 0$, $\Delta_0$ is essentially the order at which ``bad'' blocks of length $M$ start to become common.

The question of whether the annealed and quenched critical points are different has
concerned the physics community, with disagreeing predictions. Based on 
nonrigorous expansions and renormalization techniques, it was claimed in
 \cite{FLNO88} that when $c=3/2$ and $\varphi$ is asymptotically constant the two critical points are equal, while in 
\cite{DHV92} it was claimed that they are different and a prediction on the gap
between them was provided.  The question was also studied numerically in \cite{NN01}.

The following is our main result.

\begin{theorem}\label{thm}
Suppose that $\bV=(V_n)_{n\geq 1}$ is a sequence of $i.i.d.$ standard Gauusian random variables.  Then, writing $u - -\frac{\beta}{2} + \Delta$,

(i) if \eqref{excursion_law} holds with $c>3/2$, then there exist $\ep_0,\ep_1>0$ such that for all $\beta,\Delta>0$ satisfying $\Delta < \ep_0\Delta_0(\beta)$ and $\beta\Delta < \ep_1$, we have $C_q(\beta,u) = 0$; therefore $u_c^q(\beta) > u_c^a(\beta)$.  If $3/2<c<2$ then there is a constant $K$ such that for all sufficiently small $\beta$ we have $\ep_0 \Delta_0 < u_c^q(\beta) - u_c^a(\beta) < K\Delta_0$.

(ii) if \eqref{excursion_law} holds with $c=3/2$ and $\varphi(n) \to 0$ as $n \to \infty$, then $u_c^q(\beta) > u_c^a(\beta)$ for all $\beta>0$.
\end{theorem}

Theorem \ref{thm}(i) improves on the recent result in \cite{DGLT07} which establishes a positive lower bound for $\Delta_c^q(\beta)$.  The lower bound in \cite{DGLT07}, however, is $o(\Delta_0)$ and therefore does not rule out the significant use of alternate strategies.  Our proof is very different from \cite{DGLT07} as well.

Theorem \ref{thm}(ii) improves a result in \cite{DGLT07} which requires $\varphi(n) = o((\log n)^{-\eta})$ for some $\eta>1/2$.  The condition $\varphi(n) \to 0$ is equivalent to $\hp_{c-1}(t) \to \infty$ as $t \to \infty$, for $\hp_{c-1}$ of \eqref{Casymp} (see \cite{Al08}.)  For $c= 3/2$ this is equivalent to the contact fraction having an infinite derivative (as a function of $\Delta$) at $\Delta = 0$; see Lemma \ref{concavity2} below.

In \cite{Al08} it is proved that for the marginal case $c=3/2$ the
disorder is irrelevant, i.e. critical points and critical exponents are the same for quenched and annealed,
as long as the slowly varying function $\varphi(\cdot)$ satisfies the condition
\begin{eqnarray*}
  \sum_{n=1}^\infty\frac{1}{n(\varphi(n))^2}<\infty.
\end{eqnarray*}
There is a gap between such $\varphi$ and those covered by Theorem \ref{thm}(ii), and this gap contains the asymptotically constant case, $\varphi(n) \to a>0$, which includes symmetric simple random walk in 1 and 3 dimensions.  As we have noted, the physics literature contains disagreeing predictions for this case.  

\section{Notation and idea of the proof.}
{\bf Idea of the proof.}
We begin with an informal outline of the proof and the introduction of some preliminary notation.

We use $\delta^* = \delta^*(\Delta)$ as an alternate notation
for the annealed contact fraction.
The {\it annealed correlation length} is defined
to be $(\beta f_a(\beta,u))^{-1}$.  The annealed free energy of \eqref{annfreeen} is given by the variational formula
\[
  \beta f_a(\beta,u) = \sup_{\delta \geq 0} (\beta\Delta\delta - \delta I_{\mE}(\delta^{-1})),
  \]
and $\delta^*$ is the value where this sup occurs \cite{AS06}.  Here $I_{\mE}$ is the large-deviations rate function of the excursion length variable $\mE$.  For $c>1$ we have
\[
  \frac{\beta\Delta\delta^*}{\beta f_a(\beta,u)} \to (c-1) \wedge 1 \quad \text{as } \beta\Delta \to 0;
  \]
this is proved in \cite{Al08} for $c<2$ and extends readily to $c \geq 2$.  
Therefore the annealed correlation length is asymptotically proportional to 
\begin{eqnarray}\label{M_def}
  M = M(\Delta)=\frac{1}{\beta\Delta\delta^*(\Delta)}.
\end{eqnarray}

In order to show that the quenched free energy is zero, we need to show that the quenched partition 
function increases at most subexponentially. To do so we need to divide the paths into classes, and control the contribution to the partition function from each class.

For a path $\bx=(x_n)_{n\leq N}$, an excursion is called {\it long} if it exceeds a certain scale $R=R(\Delta)$ (to be determined), and {\it short} otherwise. We can view excursions as open intervals in the time axis; the closed intervals between long excursions are called {\it occupied segments}, and the union of the occupied segments forms the {\it skeleton} of the path $\bx$, denoted $\mS(\bx,R)$, or just $\mS(\bx)$ if no confusion is likely.  The {\it skeleton contact fraction} of $\bx$ is the fraction of indices $i \in \mS(\bx)$ with $x_i = 0$.  We will show that attention can effectively be restricted to skeletons in which all occupied segments have length at least $ M$.
A path $\bx$ has {\it sparse returns} if the skeleton contact fraction is less than $\delta_2=\ep_2\delta^*(\Delta)$, (with $\ep_2$ small, to be determined) and {\it dense returns}, otherwise.  As we will see, sparse-return paths are exponentially rare, and even in the annealed model their contribution to the partition function does not grow exponentially.  This annealed contribution is an upper bound for the quenched case.

More precisely, for a skeleton $\mJ$ we define $|\mJ|$ to be the number of sites in $\mJ$, $m(\mJ)+1$ to be the number of occupied segments in $\mJ$,
\[
  \mW(\mJ) = \{\bx: \mS(\bx) = \mJ\}, 
  \]
\[
    \mW_-(\mJ,\delta_2) 
    = \{\bx: \mS(\bx) = \mJ \text{ and $\bx$ has sparse returns} \},
  \]
and 
\[
   \mW_+(\mJ,\delta_2) = \{\bx: \mS(\bx) = \mJ \text{ and $\bx$ has dense returns} \}.
   \]
Ideally we would like to show that $\log P^X(\mW_-(\mJ,\delta_2) \mid \mW(\mJ)) \leq -K'|\mJ|/R$ for some $K,K'$, so the contribution to the partition function from $\mW_-(\mJ,\delta_2)$ has logarithm at most
\begin{eqnarray} \label{freesparse}
  \beta\Delta\delta_2|\mJ| -\frac{K'|\mJ|}{R} + \log P^X(\mW(\mJ)).
\end{eqnarray}
The sum of the first two terms in \eqref{freesparse} is negative, if $K'$ is large and we choose $R=R(\Delta)=\frac{1}{\beta\Delta\delta_2}$.  The same is true for the sum of the first three terms, if we discard (in an appropriate sense) short occupied segments to ensure that $|\mJ|/m(\mJ)$ is large.  Therefore in this case the whole expression in \eqref{freesparse} would be negative.  We cannot actually do exactly this; we need to incorporate coarse-graining to group together similar skeletons $\mJ$ first, and there is a positive term proportional to $m(\mJ)$ in \eqref{freesparse}, but the idea is the same.

In contrast to sparse returns, the contribution from paths with dense returns cannot be handled by comparison to the annealed system.  In this case we will use {\it semianneled}
estimates. That is, we will first compute the conditional expectation of the contribution to the partition function from $\mW_+(\mJ,\delta_2)$ given a certain average value $\overline{V}^{\mJ}$
over the skeleton $\mJ$ (or more precisely, over a coarse-grained approximation to $\mJ$.)  This conditional expectation is easily shown to be 
\begin{equation} \label{condlexp}
  E^V\left[ E^X\left( e^{\beta H_N^u(\bx,\bV)}; \mW_+(\mJ,\delta_2) \right)\ \bigg|\ \olV^{\mJ} \right] =
  E^X\left[e^{\beta(\Delta+\overline{V}^{\mJ})L_N-\frac{\beta^2L_N^2}{2|\mJ|}};\,\mW_+(\mJ,\delta_2)\right].
\end{equation} 
The quadratic term $-\beta^2 L_N^2/2|\mJ|$ in the exponent in \eqref{condlexp} reflects the fact that conditioning on $\olV^{\mJ}$ reduces the exponential moment (under $E^V$) of $H_N^u(\bx,\bV)$, and this reduction increases with the skeleton contact fraction $L_N/|\mJ|$; for dense-return paths the reduction becomes large enough to be useful in establishing that the partition function grows at most subexponentially.
An annealed estimate at this point would amount to taking the expectation with respect to $\overline{V}^{\mJ}$in \eqref{condlexp}. But this would cancel the essential quadratic term.
Instead, letting $D_{\mJ}(\bx)$ denote the contact fraction within $\mJ$, we will find a function $g(\mJ,\delta)$ and a set $T_N$ of disorders such that $\liminf_NP^V(T_N)>0$ and such that for every disorder in $T_N$ and every path in $\mW_+(\mJ,\delta_2)$, we have $\beta(\Delta+\overline{V}^{\mJ})L_N \leq g(\mJ,D_{\mJ}(\bx))$.  Then also
\begin{eqnarray}\label{Tlambda}\\
  \beta(\Delta+\overline{V}^{\mJ})L_N \leq
  \lambda\beta(\Delta+\overline{V}^{\mJ})L_N +(1-\lambda)g(\mJ,D_{\mJ}(\bx)) \nonumber
\end{eqnarray}
for every $0<\lambda<1$.
The next step is to perform an annealed estimate for the (semiannealed) partition function which has the right side of \eqref{condlexp} replaced by its upper bound from \eqref{Tlambda}. 
The logarithm of the exponential moment of $\lambda\beta\overline{V}^{\mJ}L_N$ will then be $\lambda^2\beta^2L_N^2/2|\mJ|$ which now
does not fully cancel the quadratic term $-\beta^2L_N^2/2|\mJ|$ and will result in the desired control. By means of this estimate we will only be able
to say that the partition function on $\mW_+(\mJ,\delta_2)$ increases subexponentially on the set $T_N$.  But since $T_N$ has uniformly positive
probability and the quenched free energy is nonrandom off a null set of disorders, necessarily the quenched free energy will be zero.

As noted in \cite{GT06a}, for technical convenience the partition function $Z_N$ in \eqref{freeenergy} can be replaced by the constrained partition function
\begin{eqnarray}\label{constrain_partition}
Z_N^0=E^X\left[e^{\beta H_N^u(\bx,\bV)}\,\delta_0(x_N) \right]
\end{eqnarray}
as both give the same free energy and contact fraction. 

The assumption of Gaussian disorder is only used to get neat expressions, such as in \eqref{condlexp}, when one considers conditional expectations. Otherwise it does not play a significant role and so the method should  generalize to other distributions with a finite exponential moment.

{\bf Notation.} Throughout the paper, $K_i$ and $\ep_i$ represent constants which depend only on $c$ and $\varphi$ from \eqref{excursion_law}.  Define
\begin{eqnarray}\label{Rdelta}
  R(\Delta) = \frac{1}{\beta\Delta\delta_2},
\end{eqnarray}
and
$\delta_2=\ep_2\delta^*(\Delta)$ with $\ep_2$ to be specified, satsfying $\ep_2 < 1/2$ so that $2 M(\Delta) < R(\Delta)$.
For a path $\bx$ and $A \subset \RR$ we define the local time of $\bx$ in $A$ and the corresponding contact fraction:
\[  
  L_A = L_A(\bx) = \sum_{n \in A} \delta_{0}(x_n), \quad D_A = D_A(\bx) = \frac{L_A(\bx)}{|A|},
  \]
where $|A|$ denotes the number of sites in $A$.  We abbreviate $L_{[0,N]}$ as $L_N$.  For a set $A$ of nonnegative integers, we define the average disorder
\begin{eqnarray*}
\overline{V}^{A}=\frac{1}{|A|}\sum_{n\in A} V_n.
\end{eqnarray*} 
For a general subset $B$ of $\RR$, we define $\olV^B = \olV^{B \cap \ZZ}$.  For $n \geq 1$ we let
\[
  \olm(n) = E^X(\mE; \mE \leq n) = \sum_{k=1}^n k^{-c}\varphi(k).
  \]
We denote the length of the $i$th excursion from 0 for a path $\bx$ by $\mE_i = \mE_i(\bx)$ for $i \geq 1$.

Let $\Gamma,\Gamma_1,\Gamma_2$  sets of paths. We use the notation
\begin{eqnarray} \label{ZNGamma}
Z_N(\Gamma)&=&E^X\left[e^{\beta H_N^u(\bx,\bV)} \delta_\Gamma(\bx)\,\right]\\
Z_N(\Gamma_1|\Gamma_2)&=&E^X\left[e^{\beta H_N^u(\bx,\bV)} \delta_{\Gamma_1}(\bx)\,|\,\Gamma_2\,\right], \notag
\end{eqnarray} 
and similarly for $Z_N^0$.
For $a \leq b$ we can replace $\sum_{n=0}^N$ with $\sum_{n=a}^b$ in the definition \eqref{Hdef} of the Hamiltonian, and we may restrict to a set $\Gamma$ of paths as in \eqref{ZNGamma}; we denote the resulting Hamiltonian and partition function by $H_{[a,b]}^u(\bx,\bV)$ and $Z_{[a,b]}(\Gamma)$, respectively, suppressing the dependence on $(\beta,u,\bV)$ in the latter notation.

\begin{definition}
An $R$-{\it skeleton} (or just a {\it skeleton} if confusion is unlikely) in $[0,N]$ is a set of form $[0,N]\setminus\cup_{i=1}^m(a_i,b_i)$, with $m \geq 0$, $0\leq a_1<b_1<\dots<a_m<b_m\leq N$ and $b_i - a_i \geq R$ for $i=1,\dots,m$.  In this context we use the notation $b_0 = 0, a_{m+1} = N$.  We denote a generic skeleton by $\mJ$, and $m(\mJ)$ denotes the number of open intervals in $[0,N]\setminus \mJ$.  The intervals $[b_{i-1},a_i], 1 \leq i \leq m+1$, are called the {\it occupied segments} of $\mJ$.  An occupied segment is {\it short} if its length is at most $ M(\Delta)$.  $[0,a_1]$ and $[b_m,N]$ are the {\it initial} and {\it final} occupied segments, respectively, and all other occupied segments are called {\it central}.  For a skeleton $\mJ$ we then define
\[
  \mW(\mJ) = \{\bx: \mS(\bx) = \mJ\},
  \]
\[
  \mW(\mJ,\delta) = \{\bx: \mS(\bx) = \mJ, D_{\mJ}(\bx) = \delta\},
  \]
\[
  \mW_+(\mJ,\delta) = \{\bx: \mS(\bx) = \mJ, D_{\mJ}(\bx) > \delta\},
  \]
\[
  \mW_-(\mJ,\delta) = \{\bx: \mS(\bx) = \mJ, D_{\mJ}(\bx) \leq \delta\}.
  \]
A skeleton $\mJ$ and a value $\delta>0$ are called {\it compatible} if $\mW(\mJ,\delta) \neq \phi$.  For a skeleton in $[0,N]$, a compatible $\delta$ is always a rational number with denominator at most $N$.
\end{definition}

\begin{definition}
A {\it lifted skeleton}, generically denoted $\hmJ$, is a skeleton in which all central occupied segments are long.   To each skeleton $\mJ$ there corresponds a lifted skeleton $\hmL(\mJ)$, obtained by deleting from $\mJ$ all short central occupied segments.  We define the {\it lifted skeleton of} $\bx$ to be $\hmS(\bx) = \hmL(\mS(\bx))$, and form classes of paths according to the contact fraction in this lifted skeleton:
\[
  \hmW_+(\hmJ,\delta) = \{\bx: \hmS(\bx) = \hmJ, D_{\hmJ}(\bx) > \delta\},
  \]
\[
  \hmW_-(\hmJ,\delta) = \{\bx: \hmS(\bx) = \hmJ, D_{\hmJ}(\bx) \leq \delta\}.
  \]
We then define
\begin{equation} \label{Tdef}
  \mT(\delta) = \bigcup_{\hmJ} \hmW_-(\hmJ,\delta) = \{\bx: D_{\hmS(\bx)}(\bx) \leq \delta\},
  \end{equation}
\begin{equation} \label{Ddef}
    \quad \mD(\delta) =  \bigcup_{\hmJ} \hmW_+(\hmJ,\delta) = \{\bx: D_{\hmS(\bx)}(\bx) > \delta\}.
  \end{equation}
A path in $\mT(\delta_2)$ is said to have {\it sparse returns}, and a path in $\mD(\delta_2)$ said to have {\it dense returns}.
\end{definition}

When we will deal with paths having dense returns, we will need to use a coarse graining (CG) scheme, which we introduce now.

\begin{definition} \label{CGdefs}
We fix $\ep_3$, to be specified, such that $\ep_3 R(\Delta)$ is an integer.  A {\it CG block} is an interval of form $[(k-1)\ep_3 R(\Delta),k\ep_3 R(\Delta)]$ with $k \geq 1$.  A {\it CG point} is an endpoint of a CG block.  We assume that $N$ is a CG point.  A {\it CG skeleton} is a skeleton $[0,N]\setminus\cup_{i=1}^m(a_i,b_i)$ in which all $a_i,b_i$ are CG points.  We denote a generic CG skeleton by $\mJ^*$, and write $w(\mJ^*)$ for the number of CG blocks comprising $\mJ^*$.  Given a skeleton $\mJ = [0,N]\setminus\cup_{i=1}^m(a_i,b_i)$ we let $a_i^*$ and $b_i^*$ denote the smallest CG point greater than $a_i$ and the largest CG point less than $b_i$, respectively, and we define the CG skeleton $\mL^*(\mJ) = [0,N]\setminus\cup_{i=1}^m(a_i^*,b_i^*)$, which is the union of all CG blocks that intersect $\mJ$.  If $\mJ$ is an $R(\Delta)$-skeleton then $\mL^*(\mJ)$ is a $(1-2\ep_3)R(\Delta)$-skeleton.  We let $\mS^*(\bx) = \mL^* ( \mS(\bx))$.  A {\it lifted CG skeleton} is a CG skeleton of form $\mL^*(\hmJ)$ where $\hmJ$ is a lifted skeleton; we denote a generic lifted CG skeleton by $\hmJs$.  We again form classes of paths according to the contact fraction in the lifted CG skeleton:
\[
  \mW^*(\mJ^*) = \{\bx: \mS^*(\bx) = \mJ^*\},
  \]
\[
  \mW^*(\mJ^*,\delta) = \{\bx: \mS^*(\bx) = \mJ^*, D_{\mJ^*}(\bx) = \delta\},
  \]
\[
  \hmW^*(\hmJ,\delta) = \{\bx: \mS(\bx) = \hmJ, D_{\mL^*(\hmJ)}(\bx) = \delta\}.
  \]
\end{definition}

To deal with paths having sparse returns, we need a different coarse-graining scheme, as follows.

\begin{definition} \label{semiCGdefs}
With $R = R(\Delta)$, fix some (small) $\ep_4>0$ such that $(1+\ep_4)^{l_1} = R$ for some integer $l_1$, and let $l_0 = \max\{k: (1 + \ep_4)^k <  M(\Delta)/4\}$.  Define intervals 
\[
  I_{l_0} = [0,(1 + \ep_4)^{l_0}], \qquad I_k = ((1+\ep_4)^{k-1},(1+\ep_4)^k], \quad l_0 < k \leq l_1,
  \]
\[
  I_{l_1+k} = (R + (k-1)\ep_4 R,R + k\ep_4 R], \quad k \geq 1.
  \]
We write $n_k^-$ and $n_k^+$ for the smallest and largest integers, respectively, in $I_k$.  A \emph{semi-CG skeleton} is a skeleton in which each occupied segment has length in $\{n_k^+, k \geq l_0\}$.  We denote a generic semi-CG skeleton by $\mJ^s$.  Given a skeleton $\mJ = [0,N]\setminus\cup_{i=1}^m(a_i,b_i)$ we let $a_i^s = b_{i-1} + \min\{ n_k^+: b_{i-1} + n_k^+ \geq a_i \}$ and define the semi-CG skeleton $\mL^s(\mJ) = [0,N]\setminus\cup_{i=1}^m(a_i^s,b_i)$, which is the smallest semi-CG skeleton containing $\mJ$.  Note that $\mL^s(\mJ)$ is determined by specifying for each occupied segment of $\mJ$ (i) its exact starting point, and (ii) the value of $k$ for which $I_k$ contains the segment's length.  Also, if $\mJ$ is an $R$-skeleton then $\mL^s(\mJ)$ is a $((1-\ep_4)R)$-skeleton.  We let $\mS^s(\bx) = \mL^s ( \mS(\bx))$.  A \emph{lifted semi-CG skeleton} is a semi-CG skeleton of form $\mL^s(\hmJ)$ where $\hmJ$ is a lifted skeleton; we denote a generic lifted semi-CG skeleton by $\hmJ^s$.  We once more form classes of paths according to the contact fraction in the lifted semi-CG skeleton:
\[
  \mW^s(\mJ^s) = \{\bx: \mS^s(\bx) = \mJ^s\},
  \]
\[
  \mW_-^s(\hmJ^s,\delta) = \{\bx:  \hmS^s(\bx) = \hmJ^s, D_{\hmS(\bx)}(\bx) \leq \delta\}.
  \]
Note that in contrast to $\hmW^*(\hmJ,\delta)$ in Definition \ref{CGdefs}, the condition that the density of returns be at most $\delta$ here is applies to the density in $\hmS(\bx)$.
\end{definition}

\section{Paths With Dense Returns}

Recall that a path is said to have dense returns if its skeleton contact fraction fraction is greater than $\delta_2 =\ep_2\delta^*(\Delta)$.
Let also $\alpha_0 = \alpha_0(\beta\Delta)$ be given by 
\[
  E^X\left[ e^{-\alpha_0\mE}\right]= e^{-\beta\Delta}.
  \]
The following result on the concavity of the contact fraction shows the relevance of our hypotheses on $c$ and $\varphi$.

\begin{lemma} \label{concavity2}
(i) Suppose that $P^X$ satisfies \eqref{excursion_law} with $c>3/2$.  There exists $\ep_5$ as follows.  For every $K>0$ there exists $\ep>0$ such that $0<\beta\Delta<\ep_5$ and $\Delta < \ep \Delta_0(\beta)$ imply $\delta^*(\Delta) \geq K\Delta/\beta$.

(ii) Suppose that $P^X$ satisfies \eqref{excursion_law} with $c=3/2$ and $\varphi(n) \to 0$ as $n \to \infty$.  Then $\delta^*(\Delta)/\Delta \to \infty$ as $\Delta \to 0$.
\end{lemma}

For $3/2<c<2$, for small $\beta$ the condition $\Delta < \ep \Delta_0(\beta)$ will imply the condition $0<\beta\Delta<\ep_5$, while for large $\beta$ the reverse implication will hold.  In other words, for small $\beta$ the hypothesis is that $\Delta < \ep \Delta_0(\beta)$, and for large $\beta$ the hypothesis is that $\beta\Delta$ is small.

\begin{proof}[Proof of Lemma \ref{concavity2}]
{\it Case 1.} Suppose that $E^X(\mE) < \infty$, so that $c \geq 2$.  Then the transition is first order, with $1/E^X(\mE) \leq \delta^*(\Delta) \leq 1$ for all $\Delta>0$ (see \cite{Gi06}, Theorem 2.1), while $\Delta_0(\beta) \leq \beta$ for all $\beta>0$.  Hence
\[
  \frac{\delta^*(\Delta)\beta}{\Delta} \geq 
    \frac{\beta}{ E^X(\mE)\ep \Delta_0(\beta) } \geq \frac{1}{\ep E^X(\mE)},
  \]
and the result follows immediately, with $\ep_5=\infty$.

{\it Case 2.} Suppose $c<2$.  We can extend $\varphi$ from $\ZZ^+$ to $[1,\infty)$ by piecewise linearity; the result is still slowly varying.
Define $\overline{\varphi}(x) = 1/\varphi(x^{1/(c-1)})$ and let $\overline{\varphi}^*$ be a slowly varying function conjugate to $\overline{\varphi}$.  $\overline{\varphi}^*$ is characterized (up to asymptotic equivalence) by the fact that
\begin{equation} \label{conjugate}
  \overline{\varphi}^*\left( x\overline{\varphi}(x) \right) \sim \frac{1}{\overline{\varphi}(x)} \quad \text{as } x \to \infty;
  \end{equation}
see \cite{Se76}.  Then define $\hat{\varphi}(x) =\hp_{c-1}(x) = \overline{\varphi}^*(x)^{-1/(c-1)}$ and 
\[
  G_\beta(\Delta) = \frac{\delta^*(\Delta)\beta}{2\Delta} = \frac{\delta^*(\Delta)}{2\beta\Delta}\beta^2,
  \]
so $G_\beta(\Delta_0(\beta)) = 1$. From \cite{Al08} we have
\[
  \frac{\delta^*(\Delta)}{\beta\Delta} \sim K_2 (\beta\Delta)^{-(2c-3)/(c-1)} 
    \hat{\varphi}\left( \frac{1}{\beta\Delta} \right)  \quad \text{as } \beta\Delta \to 0,
    \]
so there exists $\ep_6$ such that $\beta\Delta < \ep_6$ implies
\begin{equation} \label{factorof2}
  \half K_2 (\beta\Delta)^{-(2c-3)/(c-1)} \hat{\varphi}\left( \frac{1}{\beta\Delta} \right) \leq 
    \frac{\delta^*(\Delta)}{\beta\Delta} 
    \leq 2 K_2 (\beta\Delta)^{-(2c-3)/(c-1)} \hat{\varphi}\left( \frac{1}{\beta\Delta} \right).
    \end{equation}
Under the assumptions in (ii) the exponent in \eqref{factorof2} is 0, and we have $\varphi(x) \to 0$ as $x \to \infty$ so $\overline{\varphi}(x) \to \infty$, so $\overline{\varphi}^*(x) \to 0$, so $\hat{\varphi}(x) \to \infty$.  Therefore $G_\beta(\Delta) \to \infty$ as $\Delta \to 0$ and (ii) is proved. 

  Thus suppose $3/2 < c<2$.  Since $\Delta_0(\beta) \leq \beta$, there exists $\beta_0$ such that $\beta<\beta_0$ implies $\beta\Delta_0(\beta) < \ep_6$.  Then for $\beta < \beta_0$ and $\Delta < \Delta_0(\beta)$, by \eqref{factorof2}
\begin{equation} \label{gratio}
  G_\beta(\Delta) = \frac{G_\beta(\Delta)}{G_\beta(\Delta_0)} \geq \frac{1}{4} \left( \frac{\beta\Delta}{\beta\Delta_0}
    \right)^{-(2c-3)/(c-1)} \frac{ \hat{\varphi}\left( \frac{1}{\beta\Delta} \right) }
    { \hat{\varphi}\left( \frac{1}{\beta\Delta_0} \right) }.
    \end{equation}
With a reduction in $\ep_6$ if necessary, we then have 
\[
  G_\beta(\Delta) \geq \left( \frac{\beta\Delta}{\beta\Delta_0}
    \right)^{-(2c-3)/2(c-1)},
    \]
which exceeds $K$ for small $\Delta/\Delta_0$, proving (ii) for $\beta < \beta_0$.  For $\beta \geq \beta_0$ we can use \eqref{factorof2} to conclude that if $\beta\Delta$ is less than some $\ep_7$ then we have
\[
  G_\beta(\Delta) \geq \beta_0^2 \frac{\delta^*(\Delta)}{2\beta\Delta} \geq 
    \frac{1}{4} \beta_0^2 K_2 (\beta\Delta)^{-(2c-3)/(c-1)} \hat{\varphi}\left( \frac{1}{\beta\Delta} \right) \geq K,
  \]
proving (i) for $\beta \geq \beta_0$.

{\it Case 3.} It remains to consider $c=2$ with $E^X(\mE) = \infty$.  Here to obtain a substitute for \eqref{factorof2} we need to consider the asymptotics of $\delta^*(\Delta)$ as $\beta\Delta \to 0$.  First observe that for fixed $a>1$, for large $n$,
\begin{equation} \label{phivsm}
  \varphi(n) \log a \leq \varphi(n) \sum_{\frac{n}{a}-1 \leq k \leq n-1} \frac{1}{k}
    \leq 2\sum_{\frac{n}{2}-1 \leq k \leq n-1} \frac{\varphi(k)}{k} \leq 2\olm(n),
\end{equation}
so that for sufficiently large $s$ we have
\begin{equation} \label{varphiolm}
  \varphi(s) \leq \olm(s),
  \end{equation}
and hence for small $t$,
\begin{align} \label{Masymp}
1 - M_E(-t)
 &= \sum_{k=1}^\infty (1 - e^{-tk})\frac{\varphi(k)}{k^2} \\
 &\geq \frac{t}{2} \sum_{1 \leq k \leq 1/t} \frac{\varphi(k)}{k}\notag  \\
 &=\half t \olm\left( \frac{1}{t} \right) \notag
\end{align}
and
\begin{align} \label{Masymp2}
 1 - M_E(-t)
  &\leq t \sum_{1 \leq k \leq 1/t} \frac{\varphi(k)}{k} + \sum_{k>1/t}  \frac{\varphi(k)}{k^2} \\
  &\leq t\olm\left( \frac{1}{t} \right) + 2t\varphi\left( \frac{1}{t} \right) \notag \\
  &\leq 3t\olm\left( \frac{1}{t} \right). \notag
  \end{align}
Let $\alpha_\pm = \alpha_\pm(\beta\Delta)$ be given by
\[
  1 - \half \alpha_+\olm\left( \frac{1}{\alpha_+} \right) = e^{-\beta\Delta}, 
    \quad 1 - 3\alpha_- \olm\left( \frac{1}{\alpha_-} \right) = e^{-\beta\Delta},
  \]
so that $\alpha_- \leq \alpha_0 \leq \alpha_+$ for small $\beta\Delta$, by \eqref{Masymp} and \eqref{Masymp2}.  As $\beta\Delta \to 0$ we have
\begin{equation} \label{alphambar}
  \half \alpha_+\olm\left( \frac{1}{\alpha_+} \right) \sim \beta\Delta \quad \text{or} \quad 
    \frac{1}{\alpha_+} \frac{1}{\olm\left( \frac{1}{\alpha_+} \right)} \sim \frac{1}{2\beta\Delta},
  \end{equation}
and hence 
\[
  \frac{2\beta\Delta}{\olm\left( \frac{1}{\alpha_+} \right)} \sim
    \alpha_+(\beta\Delta) \sim \frac{2\beta\Delta}
    { \left( \frac{1}{\olm} \right)^*\left( \frac{1}{\beta\Delta} \right) },
    \]
(the second equivalence being a consequence of \eqref{alphambar} and the definition \eqref{conjugate}), and then
\begin{equation} \label{olmasymp}
  \olm\left( \frac{1}{\alpha_+} \right) \sim \left( \frac{1}{\olm} \right)^*\left( \frac{1}{\beta\Delta} \right).
  \end{equation}
The same holds similarly for $\alpha_-$ in place of $\alpha_+$, and hence also for $\alpha_0$.
Also, as $\beta\Delta \to 0$,
\[
  \frac{1}{\delta^*(\Delta)} = (\log M_E)'(-\alpha_0) \sim M_E'(-\alpha_0)
    = \sum_{k=1}^\infty e^{-\alpha_0 k}\frac{\varphi(k)}{k},
    \]
which analogously to \eqref{Masymp} leads to
\[
  \frac{1}{3}  \olm\left( \frac{1}{\alpha_0} \right) \leq \frac{1}{\delta^*(\Delta)}
    \leq \olm\left( \frac{1}{\alpha_0} \right) + \varphi\left( \frac{1}{\alpha_0} \right)
    \leq 2\olm\left( \frac{1}{\alpha_0} \right),
    \]
where the last inequality follows from \eqref{varphiolm}.  This and \eqref{olmasymp} (with $\alpha_0$ in place of $\alpha_+$) show that for small $\beta\Delta$,
\begin{equation} \label{delta*c2}
  \frac{1}{4}   \left( \frac{1}{\olm} \right)^*\left( \frac{1}{\beta\Delta} \right) \leq \frac{1}{\delta^*(\Delta)}
    \leq 3 \left( \frac{1}{\olm} \right)^*\left( \frac{1}{\beta\Delta} \right).
    \end{equation}
Since $\olm$ is slowly varying, so is $(1/\olm)^*$, so we can use \eqref{delta*c2} in place of \eqref{factorof2} to prove (i) for $c=2$ with $E^X[\mE] = \infty$ in the same manner as we did for $3/2<c<2$.  
\end{proof}

Recall
that $R(\Delta)$ and $M(\Delta)$ are defined in \eqref{Rdelta}
 and \eqref{M_def}. Note that
\begin{eqnarray}\label{RvsM}
  \frac{R(\Delta)}{M(\Delta)}= \frac{\beta\Delta\delta^*(\Delta)}{\beta\Delta\delta_2}
   = \frac{1}{\ep_2}>1.
\end{eqnarray}
Let $\frac{2}{3}<\lambda<1$ and let $K_3>2$ to be specified (see Lemma \ref{Ylambdabound}.)
For fixed $\ep_2$, we take $\ep_3$ small enough so that
\begin{eqnarray}\label{h1}
  h_1=\frac{4\ep_3}{\ep_2} < \half.
\end{eqnarray}
By \eqref{RvsM} a CG block is at most $h_1/4$ fraction of a long occupied segment:
\begin{equation}\label{RvsM2}
  \ep_3 R(\Delta) = \frac{h_1}{4} M(\Delta).
  \end{equation}
Let $K_4$ satisfy
\begin{eqnarray}\label{b}
  \frac{4}{K_4\ep_2} < \frac{1}{4}
\end{eqnarray}
and 
\begin{equation} \label{K2bound2}
  \frac{1}{8}(1-\lambda)\left( \lambda - \half \right)\ep_3\ep_2 K_4 \geq K_3.
  \end{equation}
By Lemma \ref{concavity2}, for sufficiently small $\ep_0$ and $\beta\Delta$, for $\Delta < \ep_0\Delta_0(\beta)$ we have 
\begin{equation} \label{K2bound}
   \delta^*(\Delta) \geq K_4 \frac{\Delta}{\beta}.
 \end{equation}
For a lifted CG skeleton $\hmJs$ we define 
\begin{eqnarray}\label{psi_lambda}
  \psi_\lambda(\hmJs,v,\delta) = \lambda\beta(\Delta + v)\delta |\,\hmJs\,| - \half\beta^2\delta^2 |\,\hmJs\,|,
\end{eqnarray}
\begin{eqnarray}\label{g}
  g(\hmJs,\delta) = \frac{3}{4}\beta^2\delta^2 |\, \hmJs\,| - \log P^X(\mW^*(\hmJs,\delta)).
\end{eqnarray}
Observe that if $\beta(\Delta + v)\delta |\,\hmJs\,| \leq g(\hmJs,\delta)$ then $\psi_1(\hmJs,v,\delta) \leq \psi_\lambda(\hmJs,v,\delta) + (1-\lambda)g(\hmJs,\delta)$.

\begin{lemma}\label{J,J^*}
Let $\hmJ = [0,N]\setminus\cup_{i=1}^m(a_i,b_i)$ be a lifted skeleton, with $m \geq 2$. Then  
\begin{eqnarray*}
  |\,\hmJ\,| \leq |\,\mL^*(\hmJ)\,| \leq  
  (1+ h_1)\,|\,\hmJ\,|\,.
\end{eqnarray*} 
\end{lemma}

\begin{proof}
The first inequality is clear. Regarding the second one we have 
\begin{eqnarray}\label{J,J*}
  |\,\mL^*(\hmJ)\,|&=& \sum_{i=1}^{m+1} (a_i^*- b_{i-1}^*) \nonumber \\
  &\leq& (a_1 - b_0) + (a_{m+1} - b_m) + 2\ep_3 R(\Delta) 
    + \sum_{i=2}^m \big( (a_i - b_{i-1})+2\ep_3 R(\Delta) \big) \nonumber\\
  &\leq& (a_1 - b_0) + (a_{m+1} - b_m)  
    + \sum_{i=2}^m \big( (a_i - b_{i-1})+4\ep_3 R(\Delta) \big) \\
  &\leq& 
    \left(1+ \frac{4\ep_3R(\Delta)}{M(\Delta)}\right)\,\sum_{i=1}^{m+1}( a_i - b_{i-1}) \nonumber \\
  &=& \left( 1 + \frac{4\ep_3}{\ep_2 } \right) |\hmJ|, \nonumber
\end{eqnarray}
where the last equality follows from \eqref{RvsM}.
\end{proof}

The next lemma gives a uniform lower bound for the size of a set $T_N$ of disorders in which the averages over skeletons are uniformly well-controlled.

\begin{lemma}\label{TN}
There exists $\rho = \rho(\ep_3)>0$ as follows.  For the event
\begin{eqnarray*}
   T_N = \bigcap_{\hmJ^*}\,\, \bigcap_{\delta \geq (1-h_1)\delta_2}
  \left\{ (\,V_i\,)_{i\leq N}: \beta\delta(\Delta + \olV^{\hmJs})  |\,\hmJs\,| 
  \leq g(\hmJs,\delta)\right\}.
\end{eqnarray*}
we have $P^V(T_N)\geq \rho$ for all large $N$.  Here the second intersection is over $\delta$ compatible with $\hmJs$.
\end{lemma}

\begin{proof}
By \eqref{K2bound} and Lemma \ref{J,J^*}, for $\delta \geq (1-h_1)\delta_2 \geq \half\ep_2\delta^*(\Delta)$ we have
\begin{equation} \label{Deltavs}
  \beta\delta \geq \half \ep_2\beta\delta^*(\Delta) \geq \half K_4 \ep_2 \Delta,
  \end{equation}
while $|\hmJs| \geq 2\ep_3 R(\Delta)$ for all $\hmJs$.  Hence
by Chebyshev's inequality and \eqref{b} we have 
\begin{align}
  P^V&\left( \beta\delta(\Delta + \olV^{\hmJs} ) |\,\hmJs\,| > g(\hmJs,\delta) \right) \\
  &\leq
    e^{ \beta\delta\Delta |\,\hmJs\,|-g(\hmJs,\delta) }\,\,E^V\left[ e^{\beta\delta\olV^{\hmJs}  |\,\hmJs\,| } \right]
    \notag \\
  &= \exp\left(\beta\delta\Delta |\,\hmJs\,| - \frac{1}{4}\beta^2\delta^2 |\hmJs|\right) 
    P^X(\mW^*(\hmJs,\delta)) \notag \\
  &\leq \exp\left(\left(1 - \frac{1}{8}\ep_2K_4\right)\beta\Delta\delta|\hmJs|\right) 
    P^X(\mW^*(\hmJs,\delta)) \notag \\
  &\leq \exp\left( - 2 \beta\Delta\delta \ep_3R(\Delta)\right) 
    P^X(\mW^*(\hmJs,\delta)) \notag \\
  &\leq \exp(-\ep_3) P^X(\mW^*(\hmJs,\delta)). \notag
\end{align}
We now sum over $\hmJs$ and $\delta$ and take $\rho = 1 - e^{-\ep_3}$.
\end{proof}

The next step is to separate the contribution to the partition function from the short segments of the skeletons from that of the long segments. Before doing this we need some more definitions.  We use $\bx_{[a,b]}$ to denote a generic path $(x_i)_{a \leq i \leq b}$.  When confusion is unlikely, given a path $\bx = \bx_{[0,N]}$, we also let $\bx_{[a,b]}$ denote the segment of $\bx$ from $a$ to $b$.

\begin{definition} \label{QJ}
For $b-a\geq R(\Delta)$, we will denote by $\mQ_{[a,b]}$ the set of all paths $\bx_{[a,b]}$ such that
\vskip 2mm
(i) $x_{a}=x_b=0$.
\vskip 1mm
(ii) The excursion starting from $a$ and the excursion ending at $b$ (which may be the same excursion) are long.
\vskip 1mm
(iii) All the occupied segments are short.
\vskip 2mm
\noindent
The normalized partition function over the set $\mQ_{[a,b]}$ is
\begin{eqnarray*}
Q_{[a,b]}=\frac{1}{p_{b-a}} Z_{[a,b]}(\,\mQ_{[a,b]}\,).
\end{eqnarray*}
For a lifted skeleton $\hmJ = [0,N] \bs \bigcup_{i=1}^m (a_i,b_i)$ we define
\[
  Q(\hmJ) = \prod_{i=1}^m Q_{[a_i,b_i]},
  \]
which can be viewed as a factor in the total contribution to the overall partition function from skeletons $\mJ$ with $\hmL(\mJ) = \hmJ$.  Finally we let $\mY_{[a,b),r}$ denote the set of paths $\bx$ satisfying $x_a = 0$ and having no excursions longer than $r$ and starting in $[a,b]$, and $\mY_{[a,b),r}^0 = \mY_{[a,b),r} \cap \{x_b=0\}$.  We abbreviate $\mY_{[0,n),R}$ as $\mY_{n,R}$, and $\mY_{[0,n),R}^0$ as $\mY_{n,R}^0$.  If $[a,b]$ is an occupied segment in a path $\bx$, then necessarily $\bx \in \mY_{[a,b)}^0$.  
\end{definition}

Observe that 
\begin{align} \label{ZND2}
  Z_N^0(\mD(\delta_2)) &= \sum_{\hmJ} \sum_{\{\mJ: \hmL(\mJ) = \hmJ\}} Z_N^0(\mW_+(\mJ,\delta_2)) \notag \\
  &= \sum_{\hmJ} Z_N^0(\mW_+(\hmJ,\delta_2)) Q(\hmJ) \notag \\
  &\leq \sum_{\hmJ} \sum_{\delta\geq(1-h_1)\delta_2} Q(\hmJ)
    Z_N^0(\hmW^*(\hmJ,\delta) \mid \hmW^*(\hmJ,\delta)) P^X(\hmW^*(\hmJ,\delta)),
 \end{align}
where the last sum is over $\delta$ compatible with $\mL^*(\hmJ)$.
To control the growth of \eqref{ZND2}, we will need some estimates for quantities related to those appearing on the right side.  Let us start with $P^X(\mW^*(\hmJs))$.  Define
\[
  I_k^* = [(k-1)\ep_3R(\Delta),k\ep_3 R(\Delta)].
  \]
\begin{proposition}\label{P(J^*)}
Let $\ep_8>0$.  Then there exists $K_5$ such that, provided $\beta\Delta$ is sufficiently small (depending on $\ep_3,\ep_8$), for $\hmJs = [0,N] \bs \bigcup_{i-1}^m (a_i^*,b_i^*)$ a lifted CG skeleton, for the positive integers $k_i,\ell_i$ given by $a_i^* = k_i \ep_3 R(\Delta)$, $b_i^* = \ell_i \ep_3 R(\Delta)$, we have
\begin{equation} \label{produpper}
  P^X(\mW^*(\hmJs)) \leq \prod_{i=1}^m  \frac{K_5}{(\ell_i-k_i)^{(1-\ep_8)c}}.
  \end{equation}
\end{proposition}

\begin{proof}
Write $R$ for $R(\Delta)$.  We sum over the starting and ending points for
the long excursions, within the CG blocks:
\begin{eqnarray} \label{Jstarbound}
  P^X\left( \mW^*(\hmJs) \right)&=&\sum_{\hmJ: \mL^*(\hmJ)=\hmJs} P^X(\mW(\hmJ))  \nonumber \\
  &\leq& \sum_{(a_i\in I^*_{k_i},\,b_i\in I^*_{\ell_i+1})_{i\leq m}}
    \prod_{i=1}^{m+1} P^X(\mY_{[b_{i-1},a_{i}),R}^0\,
    \big|\,x_{b_{i-1}}=0)\, \prod_{i=1}^m p_{b_i-a_i} \nonumber\\
  &\leq& \sum_{(a_i\in I^*_{k_i},\,b_i\in I^*_{\ell_i+1})_{i\leq m}}
    \prod_{i=1}^{m} P^X( x_{a_i}=0 \mid \mY_{[b_{i-1},a_{i}),R} )\, p_{b_i-a_i} 
    \nonumber \\
  &=& \sum_{(a_i\in I^*_{k_i},\,b_i\in I^*_{\ell_i+1})_{i\leq m}} \prod_{i=1}^m
    P^X( x_{a_i}=0 \mid \mY_{[b_{i-1},a_{i}),R} )
    \frac{\varphi(b_i-a_i)}{(b_i-a_i)^c}.
\end{eqnarray}
We have $(\ell_i-k_i)\ep_3R\leq 
b_i-a_i \leq (\ell_i-k_i+2)\ep_3R$, so provided $\ep_3 R$ is large enough (depending on $\ep_8$), i.e. $\beta\Delta$ is small enough,  
\[
  \varphi\bigg(\ep_3 R(\ell_i - k_i) \bigg) \leq (\ell_i - k_i)^{\ep_8c} \varphi(\ep_3 R).
  \]
Therefore we can bound \eqref{Jstarbound} by
\begin{align} \label{Jstarbound2}
 &\sum_{(a_i\in I^*_{k_i},\,b_i\in I^*_{\ell_i+1})_{i\leq m}} \prod_{i=1}^m
    P^X( x_{a_i}=0 \mid \mY_{[b_{i-1},a_{i}),R} )    
    \frac{ 2\varphi(\ep_3 R) }{ (\ep_3 R)^c(\ell_i-k_i)^{(1-\ep_8)c} } \notag \\
  &\quad =\quad \prod_{i=1}^m \,\,
    \sum_{a_i\in I^*_{k_i},\,b_{i-1}\in I^*_{\ell_{i-1}+1}}
    P^X( x_{a_i}=0 \mid \mY_{[b_{i-1},a_{i}),R} ) \,
    \frac{ 2\varphi(\ep_3 R) }{ (\ep_3 R)^c(\ell_i-k_i)^{(1-\ep_8)c} }\notag\\
  &\quad =\quad  \prod_{i=1}^m \,\,
    \sum_{b_{i-1}\in I^*_{\ell_{i-1}+1}}
    E^X( L_{I^*_{k_i}} \mid \mY_{[b_{i-1},a^*_{i}),R} ) \,
    \frac{ 2\varphi(\ep_3 R) }{ (\ep_3 R)^{c}(\ell_i-k_i)^{(1-\ep_8)c} }\notag\\
  &\quad \leq \quad \prod_{i=1}^m \ep_3 R\max_{b_{i-1}\in I^*_{\ell_{i-1}+1}} 
       E^X( L_{I^*_{k_i}} \mid \mY_{[b_{i-1},a^*_{i}),R} ) \,
       \frac{ 2\varphi(\ep_3 R) }{ (\ep_3 R)^c(\ell_i-k_i)^{(1-\ep_8)c} }. \notag\\
\end{align}
We now need the bound
\begin{eqnarray}\label{mean_est}
\max_{b_{i-1}\in I^*_{\ell_{i-1}+1}} 
    E^X( L_{I^*_{k_i}} \mid \mY_{[b_{i-1},a^*_{i}),R} ) &\leq&
    E^X( L_{\ep_3 R} \mid \mY_{\ep_3 R,R} )\nonumber\\
 &=&\sum_{k} P^X( L_{\ep_3 R}\geq k \mid \mY_{\ep_3 R,R} )\nonumber\\
&\leq& \sum_{k} P^X( \max_{i\leq k}\mE_i \leq \ep_3R \mid \mY_{\ep_3 R,R} )\nonumber\\
&=& \sum_{k} \left(1-P^X( \mE > \ep_3R \mid \mE\leq R )\right)^k\nonumber\\
&\leq&\sum_k e^{-k P^X( \mE > \ep_3R \mid \mE\leq R )}\nonumber\\
&\leq&\frac{1}{P^X( \mE > \ep_3R \mid \mE\leq R )}\nonumber\\
&\leq& \frac{K_6(\ep_3 R)^{c-1}}{\varphi(\ep_3 R)}.
\end{eqnarray}
Inserting this bound into \eqref{Jstarbound} we obtain \eqref{produpper}.
\end{proof}

\begin{lemma}\label{E[Q]}
Let $0<\theta<1$.  
Provided $\ep_2 $ is sufficiently small, and $\beta\Delta$ is sufficiently small (depending on $\ep_2$), for all $a,b$ with $b-a\geq R(\Delta)$ we have
\begin{eqnarray*}
  E^V[\,Q_{[a,b]}\,] \leq \frac{1}{1-\theta}.
\end{eqnarray*}
\end{lemma}

 \begin{proof}
We write $R,M$ for $R(\Delta),M(\Delta)$, respectively.  Let $Q^k_{[a,b]}$, $k\geq 0$, be the contribution to $Q_{[a,b]}$
from paths which have $k$ short occupied segments, so that $Q^0_{[a,b]}=1$. Let us first estimate the contribution
\begin{eqnarray*}
  Q^1_{[a,b]}=\frac{1}{p_{b-a}}\sum_{ \substack{a+R\leq j_1<j_2\leq b-R\nonumber\\
    j_2\leq j_1+M}}
    p_{j_1-a} \,E^X \left[\,e^{\beta H_{[j_1,j_2]}^u(\bx,\bV)} \delta_{ \{x_{j_2}=0\} } \mid x_{j_1}= 0\,\right]\,p_{b-j_2}.
\end{eqnarray*}
Let $W_{[a,b]}$ denote the last time that the path visits zero in the interval $[a,b]$.
Using the symmetry over the indices $j_1,j_2$ we get that 
\begin{eqnarray} \label{Q1bound}
  Q^1_{[a,b]}&\leq&\frac{2}{p_{b-a}}\sum_{
    \substack{
    a+R\leq j_1<j_2\leq b-R\nonumber\\
    j_2\leq j_1+ M;\,
    j_1\leq\frac{b+a}{2}
    }}
    p_{j_1-a} \,E^X \left[\,e^{\beta H_{[j_1,j_2]}^u(\bx,\bV)} \delta_{ \{x_{j_2}=0\} } \mid x_{j_1}= 0\,\right]\,p_{b-j_2}\nonumber \\
  &\leq&\frac{2}{p_{b-a}}\sum_{
    \substack{
    a+R\leq j_1<j_2\leq b-R\nonumber\\
    j_2\leq j_1+ M;\,
    j_1\leq\frac{b+a}{2}
    }}
    p_{j_1-a} \,E^X \left[\,e^{\beta H_{[j_1,j_1+ M]}^u(\bx,\bV)} \delta_{ \{ W_{[j_1,j_1+ M]}=j_2 \} } \mid x_{j_1}= 0\,\right]
      \nonumber \\
  &&\hskip 6cm\,\cdot\frac{p_{b-j_2}}{P^X( \mE> M-j_2+j_1)}.
\end{eqnarray}
Recalling that $ M < R$, we have $b-j_2 \geq (b-a)/4$, and thus $p_{b-j_2} \leq 2p_{\lfloor (b-a)/4 \rfloor}$, for all $j_2$ appearing in the sum.  Therefore \eqref{Q1bound} yields
\begin{align}
  E^V[ Q^1_{[a,b]} ] &= \frac{4}{p_{b-a}}\frac{ p_{\lfloor (b-a)/4 \rfloor} }{P^X(\mE> M)} \notag \\
 &\qquad \cdot\sum_{
  \substack{
  a+R\leq j_1<j_2\leq b-R\nonumber\\
  j_2\leq  M\,;
  j_1\leq\frac{b+a}{2}
  }} p_{j_1-a} \,E^X E^V \left[\,e^{\beta H_{M}^u(\bx,\bV)};\,W_{[0,M]}=j_2 - j_1\,\right] \notag\\
&\leq \frac{4^{c+2}}{P^X(\mE> M)}
  E^X E^V \left[\,e^{\beta H_{M}^u(\bx,\bV)}\,\right]
  \sum_{
  a+R\leq j_1\leq\frac{b+a}{2}
  } p_{j_1-a} \\
&\leq \frac{4^{c+2}}{P^X(\mE> M)}
  E^X \left[\,e^{\beta \Delta L_{M} }\,\right] P^X(\mE>R). \notag 
\end{align}
From \cite{Al08} we have $E^X \left[\,e^{\beta \Delta L_{M} }\,\right] < e^{K_7}$ for some constant $K_7$.  Hence provided $\beta\Delta$ is small (depending on $\ep_2$), from \eqref{RvsM} we have
\begin{eqnarray*}
  E^V[\,Q^1_{[a,b]}\, ] 
  &\leq& K_8e^{K_7} \frac{ P^X(\mE>R) }{ P^X(\mE> M) }\\
  &\leq& 2K_8 e^{ K_7 } 
    \left( \frac{M}{R} \right)^{c-1} \\
  &=& 2K_8e^{K_7} \ep_2^{c-1}.
\end{eqnarray*}
We now take $\ep_2$ small enough so the last quantity is at most $\theta$.

For $y \in [a,b]$ and $k \geq 1$ let $A_y^i$ denote the event that the $i$th long excursion starting at or after $a$ ends at $y$, and let $Q_{[a,b]}^k(A_y^i)$ denote the contribution to $Q_{[a,b]}^k$ from paths in $A_y^i$.  For $k=1$ we have $Q_{[a,b]}^1(A_b^2) = Q_{[a,b]}^1$ since all contributing paths have the second excursion ending at $b$, and similarly for $k=0$ we have $Q_{[a,b]}^0(A_b^1) = Q_{[a,b]}^0 = 1$.
Thus what we have shown is that $E^V[\,Q^1_{[a,b]}(A_b^2)\, ] \leq \theta E^V[\, Q^0_{[a,b]}(A_b^1)\, ]$.
The same argument applied to the interval $[a,y]$ in place of $[a,b]$ gives $E^V[\, Q^k_{[a,b]}(A_y^2)\, ] \leq \theta E^V[\, Q^{k-1}_{[a,b]}(A_y^1)\, ]$ for all $k \geq 2$ and all $y$, so summing over $y$ and then iterating over $k$ gives 
$E^V[\,Q^k_{[a,b]}\, ]\leq \theta^k$.  Then
\begin{eqnarray*}
  E^V[\,Q_{[a,b]}\, ]\leq \sum_{k\geq 0}E^V[\,Q^k_{[a,b]}\, ]\leq\sum_{k\geq 0}\theta^k=\frac{1}{1-\theta}.
\end{eqnarray*}
\end{proof}

In bounding $Z_N^0(\mD(\delta_2))$ via \eqref{ZND2},
the crucial estimate will be on the partition function
\begin{eqnarray*}
  Z_N^0(\hmW^*(\hmJ,\delta) \mid \hmW^*(\hmJ,\delta))
    = \exp\left( Y(\delta,\hmJ)+Y'(\delta,\hmJ) \right),
\end{eqnarray*}
where
\begin{eqnarray*}
 Y(\delta,\hmJ)=\log 
   E^V\Big[\,Z_N^0(\hmW^*(\hmJ,\delta) \mid \hmW^*(\hmJ,\delta))\,\Big|\, \overline{V}\,^{\mL^*(\hmJ)} \Big]
\end{eqnarray*}
and
\begin{eqnarray*}
  Y'(\delta,\hmJ)=\log \frac{Z_N^0(\hmW^*(\hmJ,\delta) \mid \hmW^*(\hmJ,\delta))}
    {E^V\Big[\,Z_N^0(\hmW^*(\hmJ,\delta) \mid \hmW^*(\hmJ,\delta))\,\Big|\, \overline{V}\,^{\mL^*(\hmJ)} \Big]}.
\end{eqnarray*}
Recalling \eqref{psi_lambda} and \eqref{g}, define
\begin{eqnarray*}
  Y_\lambda(\delta,\hmJ)=
 \psi_\lambda(\mL^*(\hmJ),\overline{V}\,^{\mL^*(\hmJ)},\delta) 
  +(1-\lambda)g(\mL^*(\hmJ),\delta).
\end{eqnarray*}

\begin{lemma} \label{Ylambda}
For all $0<\lambda<1$, all lifted skeletons $\hmJ$ and all $\delta \geq (1-h_1)\delta_2$ compatible with $\mL^*(\hmJ)$ we have $Y(\delta,\hmJ) \leq \psi_1(\mL^*(\hmJ),\overline{V}\,^{\mL^*(\hmJ)},\delta)$, and
on the set $T_N$ we have 
\begin{eqnarray*}
  Y(\delta,\hmJ)\leq Y_\lambda(\delta,\hmJ).
\end{eqnarray*}
\end{lemma}

\begin{proof}
Conditionally on $\olV^{\mL^*(\mJ)} = v$ for some $v$, $(V_n)_{n \in \mL^*(\mJ)}$ is multivariate normal with easily calculated mean and covariance; as noted in \cite{Al08} it follows readily that
\begin{align} \label{quad}
  e^{Y(\delta,\hmJ)} &= E^V\Big[\,Z_N^0(\hmW^*(\hmJ,\delta) \mid \hmW^*(\hmJ,\delta))\,\Big|\, 
    \overline{V}\,^{\mL^*(\hmJ)}\Big] \notag\\
  &= E^X\left[ \exp\left(\,\beta(\Delta+\overline{V}\,^{\hmJs})\,L_{\hmJ} - 
    \half \frac{\beta^2L_{\hmJ}^2}{|\,\mL^*(\hmJ)\,|}    
  \right) \,\Bigg|\,
  \hmW^*(\hmJ,\delta) \right]. 
\end{align}
By Lemma \ref{J,J^*}, on the set $\hmW^*(\hmJ,\delta)$ of paths, we have
\begin{align}
  \beta(\Delta+&\overline{V}\,^{\mL^*(\hmJ)})\,L_{\hmJ} - \half\frac{\beta^2L_{\hmJ}^2}{|\,\mL^*(\hmJ)\,|}\, \notag \\
  &= \,
 \beta(\Delta+ \overline{V}\,^{\mL^*(\hmJ)}) \delta \,|\,\mL^*(\hmJ)\,| - \frac{1}{2}\beta^2\delta^2|\,  
 \mL^*(\hmJ)\,| \notag \\
  &= \psi_1(\mL^*(\hmJ),\overline{V}\,^{\mL^*(\hmJ)},\delta), \notag 
\end{align}
and it is immediate from the definitions that on the set $T_N$ of disorders, we have
\[
  \psi_1(\mL^*(\hmJ),\overline{V}\,^{\mL^*(\hmJ)},\delta) 
    \leq \psi_\lambda(\mL^*(\hmJ),\overline{V}\,^{\mL^*(\hmJ)},\delta) 
    +(1-\lambda)g(\mL^*(\hmJ),k(\delta)),
    \]
which with \eqref{quad} yields the result.
\end{proof}

Equation \eqref{ZND2} and Lemma \ref{Ylambda} together show that on the set $T_N$, $Z_N^0(\mD(\delta_2))$ is bounded above by
\begin{equation}\label{ZNL}
  Z_{N,\lambda}^0 = \sum_{\hmJ}Q(\hmJ) \sum_{\delta\geq(1-h_1)\delta_2} \exp\left(Y_\lambda(\delta,  
  \hmJ)+Y'(\delta,\hmJ)\right) P^X(\hmW^*(\hmJ,\delta)),
\end{equation}
where the second sum is over $\delta$ compatible with $\mL^*(\hmJ)$.
We will show that $E^V[\, Z_{N,\lambda}^0\, ] $ increases at most polynomially in $N$.
Now $Q(\hmJ)$ and $\exp\left(Y_\lambda(\delta,\hmJ)+Y'(\delta,\hmJ)\right)$
are independent functions of $\bV$ for fixed $\hmJ$, and $E^V( \exp(Y'(\delta,\hmJ)) \mid \olV^{\mL^*(\hmJ)} ) = 1$, so we have 
\begin{equation}\label{EZNL}
  E^V[\, Z_{N,\lambda}^0\, ] = \sum_{\hmJ} \sum_{\delta\geq(1-h_1)\delta_2}\,E^V[\,Q(\hmJ)\,]\,
  E^V\left[\, \exp\left(Y_\lambda(\delta,\hmJ) \right)\,\right] P^X(\hmW^*(\hmJ,\delta)).
\end{equation}
Moreover, recalling $w(\mL^*(\hmJ))$ from Definition \ref{CGdefs}, we have the following estimate.

\begin{lemma} \label{Ylambdabound}
Given $K_3>0$, provided $\ep_0$ is sufficiently small, for all $\delta \geq (1-h_1)\delta_2$ and $\Delta < \ep_0\Delta_0$ we have
\begin{eqnarray*}
   E^V\left[ \exp\left(Y_\lambda(\delta,\hmJ)\right)\right] \leq
  \exp\left( -K_3 w(\mL^*(\hmJ)) \right) 
    P^X\left( \mW^*(\mL^*(\hmJ),\delta) \right)^{-(1-\lambda)}.
\end{eqnarray*}
\end{lemma}

\begin{proof}
Using \eqref{Deltavs} and \eqref{K2bound2}, for $K_9 = \frac{1}{4} (1-\lambda)\left( \lambda - \half \right)$ we obtain
\begin{eqnarray*}
 &&E^V\left[ \exp\left(Y_\lambda(\delta,\hmJ)\right)\right] \nonumber\\
  &&\quad = \exp\left( \lambda \beta \delta \Delta |\, \mL^*(\hmJ) \,| + \half (\lambda^2 - 1) \beta^2 \delta^2
   | \mL^*(\hmJ) | + \frac{3}{4}(1-\lambda)\beta^2\delta^2 | \mL^*(\hmJ) | 
   \right) \nonumber \\
   &&\qquad \qquad \cdot P^X\left( \mW^*(\mL^*(\hmJ),\delta) \right)^{-(1-\lambda)} \nonumber \\
   &&\quad =  \exp\left( \left( \lambda \beta \delta \Delta - 2K_9 \beta^2 \delta^2
      \right) |\,\mL^*(\hmJ)\,| \right) P^X\left( \mW^*(\mL^*(\hmJ),\delta) \right)^{-(1-\lambda)} \nonumber \\
    &&\quad \leq \exp\left( - (\ep_2K_9K_4 -\lambda) \beta\Delta\delta |\,\mL^*(\hmJ)\,| \right) 
      P^X\left( \mW^*(\mL^*(\hmJ),\delta) \right)^{-(1-\lambda)} \nonumber \\
    &&\quad \leq \exp\left( - \half \ep_2K_9K_4 \ep_3 w(\mL^*(\hmJ)) \right) 
      P^X\left( \mW^*(\mL^*(\hmJ),\delta) \right)^{-(1-\lambda)} \nonumber \\
    &&\quad \leq \exp\left( -K_3 w(\mL^*(\hmJ)) \right) 
      P^X\left( \mW^*(\mL^*(\hmJ),\delta) \right)^{-(1-\lambda)}. \nonumber
\end{eqnarray*}
\end{proof}

One can think of $K_3$ as the ``cost per CG block'' of an occupied segment, averaged over the disorder, on the set $T_N$ of disorders where Lemma \ref{Ylambda} applies.  Lemma \ref{Ylambdabound} says that this cost can be made arbitrarily large by taking $\ep_0$ small.  By contrast, the annealed system has a bounded gain per block, because the negative term in the exponent on the right side of \eqref{quad} is absent. 

We can now conclude the following.

\begin{lemma}\label{BoundOnZNL}
Provided $\ep_2$ and then $\ep_0$ are chosen sufficiently small, for sufficiently small $\beta\Delta$, $E^V[\,Z_{N,\lambda}^0\,]$ grows at most linearly in $N$.
\end{lemma}

\begin{proof}
From \eqref{EZNL}, Lemma \ref{E[Q]} (with $\theta = 1/2$) and Lemma \ref{Ylambdabound} (with $K_3 \geq 2$ to be specified), since there are at most $N$ values of $\delta$ compatible with a given $\hmJs$, we have
\begin{align}
  E^V&[\, Z_{N,\lambda}^0\, ]  \notag \\
  &\leq \sum_{\hmJs}\, \sum_{\delta \geq (1-h_1)\delta_2}\, \sum_{\hmJ: \mL^*(\hmJ) = \hmJs} 
    2^{m(\hmJ)}\, \exp\left( -K_3 w(\hmJs) \right) \notag \\
  &\qquad \qquad \qquad \qquad \cdot 
    P^X\left( \mW^*( \hmJs,\delta ) \right)^{-(1-\lambda)} P^X(\hmW^*(\hmJ,\delta)) \notag \\
  &\leq \sum_{\hmJs}
    2^{w(\hmJs)}\, \exp\left( -K_3 w(\hmJs) \right) \sum_{\delta \geq (1-h_1)\delta_2}
    P^X\left( \mW^*(\hmJs,\delta) \right)^{\lambda} \notag \\
  &\leq N \sum_{l=1}^\infty e^{-K_3l/2} \sum_{\{\hmJs: w(\hmJs)=l\}}  
    P^X\left( \mW^*(\hmJs) \right)^{\lambda}. \notag
  \end{align}
Since $\lambda>2/3$, we can take $\ep_8$ so that $(1-\ep_8)\lambda c>1$.  For fixed $l$ a CG skeleton $\hmJs$ with $w(\hmJs)=l$ can be characterized by a sequence of $l-1$ positive integers, the $j$th integer giving the number of CG blocks from the $j$th CG block in $\hmJs$ to the $(j+1)$st CG block in $\hmJs$.  Therefore by Proposition \ref{P(J^*)} we can take $K_3$ such that
\[
  \sum_{\{\hmJs: w(\hmJs)=l\}} P^X\left( \mW^*(\hmJs) \right)^{\lambda}
    \leq K_5^{\lambda (l-1)} \left( 1 + \sum_{j=2}^\infty \frac{1}{j^{(1-\ep_8)\lambda c} } \right)^{l-1}
    \leq e^{K_3 l/4}.
  \]
Then
\[
  E^V[\, Z_{N,\lambda}^0\, ]  \leq N  \sum_{l=1}^\infty e^{-K_3l/4} \leq 3Ne^{-K_3/4}.
    \]
\end{proof}  

The following is straightforward from Lemma \ref{BoundOnZNL}, Cheyshev's inequality and the Borel-Cantelli lemma.

\begin{proposition}\label{a.s. decay}
Provided $\ep_2$ and then $\ep_0$ are chosen sufficiently small, for sufficiently small $\beta\Delta$, with $P^V$ probability one, we have
\begin{eqnarray*}
\limsup_{N\to\infty}\frac{1}{N}\log Z_{N,\lambda}^0=0.
\end{eqnarray*}
\end{proposition}

\section{Paths With Sparse Returns}
We estimate $Z_N^0(\mT(\delta_2))$ using the following variant of \eqref{ZND2}:
\begin{align} \label{ZNT}
  E^V&[ Z_N^0(\mT(\delta_2)) ] \notag \\
  &= \sum_{\hmJ^s} \sum_{ \{\hmJ: \mL^s(\hmJ) = \hmJ^s\} } 
    E^V[ Q(\hmJ) ]\, E^V[ Z_N^0(\mW_-(\hmJ,\delta_2)) ] \notag \\
  &\leq \sum_{\hmJ^s} E^V[ Z_N^0(\mW_-^s(\hmJ^s,\delta_2)) ] \max_{ \{\hmJ: \mL^s(\hmJ) = \hmJ^s\} } 
    E^V[ Q(\hmJ) ] \\
  &= \sum_{\hmJ^s} E^V\left[ Z_N^0\big(\mW_-^s(\hmJ^s,\delta_2) \mid \mW^s(\hmJ^s)\big) \right]\ P^X( \mW^s(\hmJ^s) )
    \max_{ \{\hmJ: \mL^s(\hmJ) = \hmJ^s\} } E^V[ Q(\hmJ) ]. \notag
\end{align}
The last maximum is easily bounded: by Lemma \ref{E[Q]} with $\theta= 1/2$ we have 
\begin{equation} \label{Prop3.6use}
  \max_{ \{\hmJ: \mL^s(\hmJ) = \hmJ^s\} } E^V[ Q(\hmJ) ] \leq 2^{m(\hmJ^s)}.
  \end{equation}
By straightforward computation (cf. \eqref{annealed_measure}, \eqref{EVEX}), for a lifted skeleton $\hmJ$, we have the annealed bound
\begin{equation}\label{annealed1}
 E^V\left[ Z_N^0\big(\mW_-^s(\hmJ^s,\delta_2) \mid \mW^s(\hmJ^s)\big) \right] 
    \leq \exp\left( \beta\Delta\delta_2\,|\,\hmJ^s\,| \right) 
    P^X\left( \mW_-^s(\hmJ^s,\delta_2) \mid \mW^s(\hmJ^s) \right),
\end{equation}
so we need to show that, on the right side, the exponential decay of the probability overcomes the growth of the exponential factor.  

We truncate and tilt the excursion length distribution to obtain a measure $\nu_{\alpha,R}$ on paths, given by
\begin{eqnarray}\label{nu}
  \nu_{\alpha,R}(\mE = k)=\frac{e^{\alpha k}}{E^X\left[\,e^{\alpha\mE}\,|\,\mE \leq R\,\right]}\,P^X\left(\mE=k\,\big|\,  
  \mE \leq R\right), \quad k \geq 1.
\end{eqnarray}
(Strictly speaking, $\nu_{\alpha,R}$ specifies a distribution only for excursion lengths, not for paths, but since the only relevant feature of the paths is their returns to 0, we will mildly abuse notation and view $\nu_{\alpha,R}$ as a distribution on paths.) 
We then have the following.

\begin{lemma}\label{LDP}
For all $\beta, \chi, n, R$ positive, for $\mY_{n,R}, \mY_{n,R}^0$ from Definition \ref{QJ}, we have
\begin{eqnarray*}
E^X\left[\,e^{-\beta\chi L_n}\,\big|\,\mY_{n,R}^0 \,\right]=e^{-\alpha n}\frac{\nu_{\alpha,R}(x_n=0)}{P^X(\,x_n=0\,|\,\mY_{n,R})},
\end{eqnarray*}
where $\alpha=\alpha(\beta\chi,R)$ satisfies  
\begin{eqnarray}\label{alpha_def}
  e^{\beta\chi}=E^X\left[\,e^{\alpha\mE}\,|\,\mE \leq R\,\right].
\end{eqnarray}
\end{lemma}

\begin{proof}
We compute
\begin{align*}
E^X&\left[\,e^{-\beta\chi L_n}\,\big|\, \mY_{n,R}^0\,\right]\\
&= \frac{E^X\left[\,e^{-\beta\chi L_n} \delta_{ \{x_n=0\} }\,|\,\mY_{n,R}\,\right]}{P^X(\,x_n=0\,|\, \mY_{n,R}\,)}\\
&= \frac{\sum_{k} e^{-\beta\chi k} P^X\left(\, \mE_1+\cdots+\mE_k=n,\,|\, \mY_{n,R}\,\right)}{P^X(\,x_n=0\,|\,  
  \mY_{n,R}\,)}\\
&= \frac{e^{-\alpha n}}{P^X(\,x_n=0\,|\, \mY_{n,R}\,)}
\sum_{k}\frac{ e^{\alpha n} P^X\left(\, \mE_1+\cdots+\mE_k=n\,|\, \mY_{n,R}\,\right)}
{\left(E^X\left[\,e^{\alpha\mE}\,|\,\mE \leq R\,\right]\right)^k}\\
&= e^{-\alpha n}\,\frac{\nu_{\alpha,R}(\,x_n=0\,)}{P^X(\,x_n=0\,|\, \mY_{n,R}\,)}.
\end{align*}
\end{proof}

The ratio of return probabilities which appears in Lemma \ref{LDP} is difficult to bound uniformly in $n,R$.  The purpose of our semi-CG skeletons is to allow replacement of the return probabilities at time $n$ by expected numbers of returns in an interval $I_i$ (see Definition \ref{semiCGdefs}.)  These are more readily estimated, as follows.  

\begin{lemma} \label{ratio_bound}
Let $l_0$ be as in Definition \ref{semiCGdefs}.
There exists $K_{10}$ (depending on $\ep_4$) such that, provided $R$ is sufficiently large, for all $\alpha>0$ and $i \geq l_0$,
\begin{equation} \label{ratiobound3}
  \frac{ E_{\nu_{\alpha,R}}(\, L_{I_i}\,)}{E^X(\, L_{I_i}\,|\, \mY_{n_i^+,R}\,)} \leq K_{10}.
  \end{equation}
\end{lemma}

\begin{proof}
For $i=l_0$ the lemma (with $K_{10}=1$) follows from the fact that excursion lengths are stochastically larger under $\nu_{\alpha,R}$ than under $P^X( \cdot \,|\, \mY_{n_i^+,R}\,)$.  Hence we 
fix $i>l_0, R$ and $\alpha$ and define $n = n_i^+ \wedge R$.
Let $s_i = n_i^+ - n_i^- \in (\ep_4 n - 2,\ep_4 n]$ and let $r_i = \lfloor s_i/4 \rfloor + 1 \geq \ep_4 n/4$.
For the numerator of \eqref{ratiobound3}, using again the stochastic domination of excursion lengths we have 
\begin{align} \label{numerator1}
  E_{\nu_{\alpha,R}}(\, L_{I_i}\,) & \leq E_{\nu_{\alpha,R}}(\, L_{I_i} \mid x_{n_i^-} = 0\,) \\
  &\leq E_{\nu_{\alpha,R}}( L_{s_i} ) \notag \\
  &\leq 4 E_{\nu_{\alpha,R}}( L_{r_i} ) \notag \\
  &\leq 4E^X(\, L_{r_i}\,|\, \mY_{r_i,R}\,). \notag
  \end{align}
  
For the denominator or \eqref{ratiobound3}, let $J_2 = (n_i^-,n_i^- + 2r_i]$, which is roughly the first half of $I_i$, let $n = n_i^+ \wedge R$, and let $\eta_{J_2} = \inf\{t \geq 0: x_t \in J_2\}$, with $\eta_{J_2} = \infty$ if there is no such $t$.
If we condition in the denominator also on a return to 0 in $J_2$, then we get a lower bound similar to \eqref{numerator1}.  More precisely, we have
\begin{align} 
  E^X(\, L_{I_i}\,|\, \mY_{n_i^+,R}\,) &\geq \sum_{j \in J_2} E^X( L_{I_i} \mid  \mY_{n_i^+,R}, \eta_{J_2} = j )
    P^X(\, \eta_{J_2} = j \,|\, \mY_{n_i^+,R}\,) \\
  &\geq E^X(\, L_{r_i}\,|\, \mY_{r_i,R}\,) P^X(\eta_{J_2} \in J_2 \mid  \mY_{n_i^+,R}), \notag
  \end{align}
which with \eqref{numerator1} shows that
\begin{equation} \label{ratiobound4}
  \frac{ E_{\nu_{\alpha,R}}(\, L_{I_i}\,)}{E^X(\, L_{I_i}\,|\, \mY_{n_i^+,R}\,)} \leq \frac{4}{ P^X(\eta_{J_2} \in J_2 \mid  \mY_{n_i^+,R}) },
\end{equation}
so we need a lower bound for the probability on the right side of \eqref{ratiobound4}.   

Define the interval
\[
  J_1 = \left( n_i^- - n, n_i^- \right]\ \cap\ [0,\infty).
    \]
Note that $J_1$ and $J_2$ are adjacent.  Due to the truncation of excursion lengths, there is always a visit to 0 in $J_1$, provided we count the visit at time 0 when $0 \in J_1$, and considering the first such return we obtain
\begin{equation} \label{hitspot}
  P^X(\eta_{J_2} \in J_2 \mid  \mY_{n_i^+,R}) \geq \min_{j \in J_1} P^X(\eta_{J_2} \in J_2 \mid  \mY_{n_i^+,R} \cap \{x_j = 0\} ).
  \end{equation}
If $E^X(\mE) < \infty$, it follows easily from the SLLN that the right side of \eqref{hitspot} is near 1 provided $n_i^+$ is large, so we assume $E^X(\mE) = \infty$, which means $c \leq 2$.

For $j \geq 0$ let $U_j$ and $W_j$ be the starting and ending points, respectively, for the first excursion starting in $[j,\infty)$ of length at least $r_i$.  If $U_j > n_i^-$ then there is no excursion which jumps over the interval $J_2$, so $\eta_{J_2} \in J_2$.  Hence
for $j \in J_1$,
\begin{align} \label{denomcalc}
  P^X&(\eta_{J_2} \in J_2 \mid  \mY_{n_i^+,R} \cap \{x_j = 0\} ) \notag \\
  &= P^X(\eta_{J_2} \in J_2 \mid  \mY_{\infty,R} \cap \{x_j = 0\} ) \notag \\
  &\geq P^X( U_j >n_i^- \mid  \mY_{\infty,R} \cap \{x_j = 0\} ) \notag \\
  &\qquad  + \sum_{k \in [j,n_i^-]} P^X( U_j = k, W_j \in J_2 \mid  \mY_{\infty,R} \cap \{x_j = 0\} ) \notag \\
  &\geq P^X( U_j >n_i^- \mid  \mY_{\infty,R} \cap \{x_j = 0\} ) \notag \\
  &\qquad  + \sum_{k \in [j,n_i^-]} P^X( W_j \in J_2 \mid  \mY_{\infty,R} \cap \{x_j = 0\} \cap \{ U_j = k\}) \notag \\
  &\qquad \qquad \qquad \qquad \cdot P^X(U_j = k \mid  \mY_{\infty,R} \cap \{x_j = 0\} ) \notag \\
  &= P^X( U_j >n_i^- \mid  \mY_{\infty,R} \cap \{x_j = 0\} ) \notag \\
  &\qquad  + \min_{k \in [j,n_i^-]}  P^X( \mE \in (n_i^- - k,n_i^- - k + 2r_i] \mid \mE \geq r_i)\notag \\
  &\qquad \qquad \qquad \qquad \cdot P^X(U_j \leq n_i^- \mid  \mY_{\infty,R} \cap \{x_j = 0\} )\notag \\
  &\geq \min_{k \in [j,n_i^-]}  \frac{ P^X( \mE \in [n_i^- - k + r_i,n_i^- - k + 2r_i]) }{ P^X(\mE \geq r_i) }.
  \end{align}
Since $n_i^- - k + 2r_i \leq 2n$, provided $n$ is large (depending on $\ep_4$), the last ratio in \eqref{denomcalc} is bounded below by
\[
  \frac{1-c}{2}\, \frac{ r_i (2n)^{-c} \varphi(2n) }{ r_i^{1-c} \varphi(r_i) } \geq \frac{1-c}{2}\, \left( \frac{\ep_4}{2} \right)^c 
    \frac{\varphi(2n)}{\varphi(r_i)} \geq K_{11}.
  \]
With $K_{10} = 4/K_{11}$, the lemma follows from this together with \eqref{ratiobound4}, \eqref{hitspot} and \eqref{denomcalc}.
\end{proof}

\begin{lemma}\label{alpha_relation}
Let $K_{12}>0$, let $\alpha = K_{12}/R(\Delta)$ and let $\chi$ be given by \eqref{alpha_def}.  Provided $\ep_2$ is sufficiently small (depending on $K_{12}$) and $\beta\Delta$ is sufficiently small (depending on $\ep_2$), we have 
\begin{equation}\label{alpha_relation2}
   \beta\chi\delta_2 \leq \frac{1}{2}\alpha(\beta\chi,R(\Delta)).
   \end{equation}
\end{lemma}

\begin{proof}
We write $\alpha$ for $\alpha(\beta\chi,R)$, $\delta^*$ for $\delta^*(\Delta)$, $M$ for $M(\Delta)$ and $R$ for $R(\Delta)$.  We have 
\[
  e^{\alpha k} \leq 1 + \frac{e^{K_{12}}-1}{K_{12}}\alpha k \quad \text{for all } k \leq R,
  \]
so 
\[
  e^{\beta\chi} = E^X\left( e^{\alpha \mE} \mid \mE \leq R \right) \leq
    1 +  \frac{e^{K_{12}}-1}{K_{12}}\alpha E^X(\mE \mid \mE \leq R),
    \]
and therefore, for large $R$,
\begin{equation} \label{bcupper}
  \beta\chi\delta_2 \leq  2\frac{e^{K_{12}}-1}{K_{12}}\alpha \olm(R)\delta_2.
  \end{equation}
Hence we need to show that
\begin{equation} \label{mbardelta2}
  \olm(R)\delta_2 = \frac{\olm(R)}{R}\ \frac{1}{\beta\Delta} \leq \frac{K_{12}}{4(e^{K_{12}} - 1)}.
  \end{equation}

{\it Case 1.} When $E^X(\mE) < \infty$, \eqref{mbardelta2} is true whenever $\ep_2$ is small, since $\olm(R) \leq E^X(\mE)$.  

{\it Case 2.} Suppose $3/2 \leq c<2$.  Then as $\beta\Delta \to 0$, for some $K_{13},K_{14}$ we have  
\[
  \frac{\olm(R)}{R} \sim K_{13} \left(\frac{\ep_2}{M} \right)^{c-1} \varphi\left( \frac{M}{\ep_2} \right) \quad \text{and}
    \quad \frac{1}{\beta\Delta} \sim K_{14} \frac{M^{c-1}}{\varphi(M)},
  \]
the first being uniform in $\ep_2<1$ and the second being proved in \cite{Al08}.  Hence for $\beta\Delta$ small (so that $M$ is large) and $\ep_2<1$ we have 
\[
  \olm(R)\delta_2 \leq K_{15} \ep_2^{c-1} \varphi\left( \frac{M}{\ep_2} \right) \varphi(M)^{-1}
    \leq K_{15}\ep_2^{(c-1)/2},
  \]
and \eqref{mbardelta2} follows for small $\ep_2$.

{\it Case 3.} Suppose $c=2$ and $E^X(\mE) = \infty$.  For $\beta\Delta$ small and $\ep_2 < 1$ we obtain using \eqref{delta*c2} that
\begin{align} \label{c2case}
  \olm(R)\delta_2 = \ep_2 \olm\left( \frac{1}{\ep_2\beta\Delta\delta^*} \right) \delta^*
    \leq 2\ep_2^{1/2} \frac{  \olm\left( \frac{1}{\beta\Delta} \left( \frac{1}{\olm} \right)^*
    \left( \frac{1}{\beta\Delta} \right)\right) }
    { \left( \frac{1}{\olm} \right)^*\left( \frac{1}{\beta\Delta} \right) }
    \leq 4\ep_2^{1/2},
  \end{align}
and \eqref{mbardelta2} follows for small $\ep_2$.  Here we use the fact that the rightmost ratio in \eqref{c2case} converges to 1 as $\beta\Delta \to 0$ by the definition of the conjugate.
\end{proof}

The next lemma shows that cost per length $R(\Delta)$, in occupied segments, of having sparse returns can be made arbitrarily large by taking $\ep_2$ small.  This cost appears as the constant $K_{16}$.

\begin{lemma}\label{LDP_Lower}
There exists a constant $C = C(R(\Delta))$ as follows.
For every $K_{16}>0$, provided $\ep_2$ is small enough (depending on $K_{16}$ and $\ep_4$), for all lifted semi-CG skeletons $\hmJ^s$, for $K_{10}$ from Lemma \ref{ratio_bound},
\begin{eqnarray*}
  P^X\left( \mW_-^s(\hmJ^s,\delta_2) \mid \mW^s(\hmJ^s) \right)
    \leq C(4K_{10})^{m(\hmJ^s)+1} e^{-K_{16}\,|\,\hmJ\,|/R(\Delta)}.
\end{eqnarray*} 
\end{lemma}

\begin{proof}
We write $R,M$ for $R(\Delta),M(\Delta)$.  Let $\alpha = K_{16}/(1-\ep_4)R(\Delta)$.

Fix $\hmJ^s = [0,N] \bs \cup_{i=1}^m (a_i,b_i)$for which $P^X(\mW^s( \hmJ^s )) > 0$, and for paths in $\mW^s(\hmJ^s)$ let $[b_{i-1},b_{i-1} + T_i]$ denote the $i$th occupied segment, for $1 \leq i \leq m+1$.  Thus, given $\mW^s(\hmJ^s)$, $b_{i-1}$ is deterministic and $T_i$ is random, but lying in $I_i' = I_{k_i} \cap \ZZ$ for some particular $k_i = k_i(\hmJ^s)$, for $i \leq m$, and $T_{m+1} = N - b_m$.  For notational convenience we define $t_{m+1} = N - b_m$ and the one-point interval $I_{m+1}' = \{t_{m+1}\}$.  

Suppose $t_i \in I_i'$ for all $i \leq m+1$.  If there is at least one long occupied segment in $\hmJ^s$ then it follows analogously to Lemma \ref{J,J^*} that $\sum_{i=1}^{m+1} t_i \geq (1-\ep_4)||\hmJ^s|| -  M/4 \geq |\hmJ^s|/2$. If there is no long occupied segment, then $m=1$ and $t_1+t_2 < 2 M$, and we have $\alpha |\hmJ^s| < 3\alpha  M \leq 6\ep_2  K_{16} \leq 6K_{16}$.
Thus is all cases we have
\begin{equation} \label{tisum}
  \exp\left( -\frac{\alpha}{2} \sum_{i=1}^{m+1} t_i \right) \leq e^{2K_{16}} e^{-\alpha |\hmJ^s|/4}.
\end{equation}

The $T_i$'s are conditionally independent given $\mW^s(\hmJ^s)$; in fact
\begin{align} \label{hitpoint}
  P^X&\left( T_i = t_i \text{ for all } i\leq m\ \big|\ \mW^s(\hmJ^s) \right) \\
  &= \prod_{i=1}^{m} P^X\left( T_i = t_i\ \big|\ \mW^s(\hmJ^s) \right) \notag \\
  &= \prod_{i=1}^{m} \frac{ P^X\left( x_{t_i} = 0, \mY_{t_i,R} \right) 
    P^X\left( \mE = b_i - b_{i-1} - t_i\right) }{ \sum_{t \in I_i'} 
    P^X\left( x_t = 0, \mY_{t,R} \right) 
    P^X\left( \mE = b_i - b_{i-1} - t \right) }. \notag 
    \end{align}
For sites $t \in I_i'$ we have $b_i - b_{i-1} - t \geq (1-\ep_4)R$, while $||I_i'|| \leq \ep_4 R$.
Therefore provided $\ep_4$ is small,
\begin{align}
  \max_{s,t \in I_i'} \frac{ P^X\left( \mE = b_i - b_{i-1} - s \right) }{ P^X\left( \mE = b_i - b_{i-1} - t \right) }
    &\leq 2,
\end{align}
which with \eqref{hitpoint} shows that
\begin{align} \label{noweight}
  P^X&\left( T_i = t_i \text{ for all } i\leq m\ \big|\ \mW^s(\hmJ^s) \right)
    \leq  \prod_{i=1}^{m} \frac{ 2P^X\left( x_{t_i} = 0,  \mY_{t_i,R} \right) }{ \sum_{t \in I_i'} 
    P^X\left( x_t = 0,  \mY_{t,R} \right) }. 
  \end{align}
Let $\alpha = K_{16}/(1-\ep_4)R(\Delta)$ and let $\chi$ be given by \eqref{alpha_def}.
We obtain using \eqref{tisum}, \eqref{noweight} and Lemmas \ref{LDP} and \ref{alpha_relation} that
\begin{align} \label{Wprob}
  P^X&\left( \mW_-^s(\hmJ^s,\delta_2) \mid \mW^s(\hmJ^s) \right) \notag \\
  &= \sum_{t_1 \in I_1'} \cdots \sum_{t_m \in I_m'}
    P^X\left( \frac{ L_{\mS^s(\bx)} }{ |\mS(\bx)| } \leq \delta_2\ \bigg|\ \mW^s(\hmJ^s), 
    T_i = t_i \text{ for all } i \leq m \right) \notag \\
  &\qquad \qquad \qquad \qquad \qquad \qquad \cdot P^X\left( T_i = t_i \text{ for all } i \leq m\ \big|\  
    \mW^s(\hmJ^s) \right) \notag \\
  &\leq \sum_{t_1 \in I_1'} \cdots \sum_{t_m \in I_m'}
    \prod_{i=1}^{m+1} e^{\beta\chi\delta_2 t_i}
    E^X\left( e^{-\beta\chi L_{t_i}}\ \big|\ x_{t_i} = 0, \mY_{t_i,R} \right) \notag \\
  &\qquad \qquad \qquad \qquad \qquad \qquad \cdot 
    \frac{ 2P^X\left( x_{t_i} = 0,  \mY_{t_i,R} \right) }{ \sum_{t \in I_i'} 
    P^X\left( x_t = 0,  \mY_{t,R} \right) } \notag \\
  &= \sum_{t_1 \in I_1'} \cdots \sum_{t_m \in I_m'}
    \prod_{i=1}^{m+1} e^{-\alpha t_i/2} 
    \frac{ \nu_{\alpha,R}(x_{t_i} = 0) }{ P^X(x_{t_i} = 0 \mid \mY_{t_i,R}) } \\
  &\qquad \qquad \qquad \qquad \qquad \qquad \qquad \cdot
     \frac{ 2P^X\left( x_{t_i} = 0, \mY_{t_i,R} \right) }{ \sum_{t \in I_i'} 
    P^X\left( x_t = 0, \mY_{t,R} \right) } \notag \\
  &\leq e^{2K_{16}} e^{-\alpha |\hmJ^s|/4} \prod_{i=1}^{m+1}  
    \frac{  2\sum_{t \in I_i'}  \nu_{\alpha,R}(x_{t} = 0) P^X( \mY_{t,R}) }
    { \sum_{t \in I_i'} P^X(x_t = 0 \mid  \mY_{t,R}) P^X( \mY_{t,R}) }. \notag 
  \end{align}
Now the event $\mY_{t,R}$ is nonincreasing in $t$, so 
\begin{equation} \label{Ycond}
  \max_{s,t \in I_i'} \frac{ P^X( \mY_{s,R}) }{ P^X( \mY_{t,R}) } \leq \frac{ P^X( \mY_{n_{k_i}^-,R}) }{ P^X( \mY_{n_{k_i}^+,R}) }
    = \frac{1}{ P^X( \mY_{n_{k_i}^+,R} \mid \mY_{n_{k_i}^-,R}) }
    \leq \frac{1}{ P^X( \mY_{n_{k_i}^+ - n_{k_i}^-,R}) }.
  \end{equation}
Since $n_{k_i}^+ - n_{k_i}^- \leq \ep_4 R$, it is easily shown that provided $\ep_4$ is sufficiently small (and also $\ep_2 \leq \ep_4$, to cover the case of $k_i=l_0$ which occurs if the initial segment is short), the denominator on the right side of \eqref{Ycond} is at least 1/2, and therefore
\begin{equation} \label{Ycond2}
  \max_{s,t \in I_i'} \frac{ P^X( \mY_{s,R}) }{ P^X( \mY_{t,R}) } \leq 2.
  \end{equation}
Further, for fixed $t$, $P^X(x_t = 0 \mid  \mY_{j,R})$ takes the same value for all $j \geq t$.  With \eqref{Wprob}, \eqref{Ycond2} and Lemma \ref{ratio_bound} this shows that
\begin{align} \label{conseq}
  P^X&\left( \mW_-^s(\hmJ^s,\delta_2) \mid \mW^s(\hmJ^s) \right) \notag \\
  &\leq  e^{2K_{16}} e^{-\alpha |\hmJ^s|/4} \prod_{i=1}^{m+1}  
    \frac{  4\sum_{t \in I_i'}  \nu_{\alpha,R}(x_{t} = 0) }
    { \sum_{t \in I_i'} P^X(x_t = 0 \mid  \mY_{t,R}) }  \\
  &=  4^{m+1} e^{2K_{16}} e^{-\alpha |\hmJ^s|/4} \frac{  \nu_{\alpha,R}(x_{t_{m+1}} = 0) }
    { P^X(x_{t_{m+1}} = 0 \mid  \mY_{t_{m+1},R}) } 
    \prod_{i=1}^m  \frac{ E_{\nu_{\alpha,R}}(\, L_{I_i'}\,)}{E^X(\, L_{I_i'}\,|\, \mY_{n^+,R}\,)} \notag \\
  &\leq 4^{m+1} K_{10}^m e^{-\alpha |\hmJ^s|/4}  \frac{ e^{2K_{16}} }
    { P^X(x_{t_{m+1}} = 0 \mid  \mY_{t_{m+1},R}) }. \notag
  \end{align}
For large $R$, since $\lim_{t \to \infty} P^X(x_t = 0 \mid  \mY_{t,R}) = E^X( \mE \mid \mE \leq R)^{-1} > 0$, there exists $C(R)>0$ such that $P^X(x_t = 0 \mid  \mY_{t,R}) \geq C(R)^{-1}$ for all $t$ for which $P^X(x_t = 0 \mid  \mY_{t,R})>0$.  Since $P^X(\mW^s( \hmJ^s )) > 0$, $t_{m+1}$ must be such a value of $t$, so from \eqref{conseq},
\[
  P^X\left( \mW_-^s(\hmJ^s,\delta_2) \mid \mW^s(\hmJ^s) \right) \leq C(R)(4K_{10})^{m+1} e^{-K_{16} |\hmJ^s|/R}.
  \]
\end{proof}

\begin{proposition} \label{ZNbound}
Provided $\ep_2$ is sufficiently small, and $\beta\Delta$ is sufficiently small (depending on $\ep_2$),
\begin{eqnarray} \label{Tbound}
  E^V[\,Z_N^0(\mT(\delta_2))\,] \quad \text{is bounded in } N.
\end{eqnarray}
\end{proposition}

\begin{proof}
Write $R$ for $R(\Delta)$.  From \eqref{ZNT}--\eqref{annealed1}, and from Lemma \ref{LDP_Lower} with $K_{16} \geq 2$ to be specified, for $C(R)$ and $K_{16}$ from that lemma, provided $\beta\Delta$ is sufficiently small we have
\begin{align} \label{ZNT2}
  E^V&[ Z_N^0(\mT(\delta_2)) ] \notag \\
  &\leq \sum_{\hmJ^s} 2^{m(\hmJ^s)} \exp\left( \beta\Delta\delta_2\,|\,\hmJ^s\,| \right) 
    P^X\left( \mW_-^s(\hmJ^s,\delta_2) \mid \mW^s(\hmJ^s) \right)\ P^X( \mW^s(\hmJ^s) ) \notag \\
  &\leq C(R) \sum_{\hmJ^s} \exp\left( \left( \beta\Delta\delta_2 
    - \frac{ K_{16} }{ R } \right) |\,\hmJ^s\,| \right) 
    (4K_{10})^{m(\hmJ^s)+1} \ P^X( \mW^s(\hmJ^s) ) \\
  &\leq C(R) \sum_{\hmJ^s} \exp\left( - \frac{ K_{16} }{ 2R } |\,\hmJ^s\,| \right) 
    (4K_{10})^{m(\hmJ^s)+1} \ P^X( \mW^s(\hmJ^s) ). \notag
\end{align}
We use notation from the proof of Lemma \ref{LDP_Lower}.  In particular,
for a lifted semi-CG skeleton $\hmJ^s = [0,N] \bs \cup_{i=1}^m (a_i,b_i)$for which $P^X(\mW^s( \hmJ^s )) > 0$, and for paths in $\mW^s(\hmJ^s)$ let $[b_{i-1},b_{i-1} + T_i]$ denote the $i$th occupied segment, for $i \leq m+1$.  $T_i$ is then required to lie in $I_i' = I_{k_i} \cap \ZZ$ for some particular $k_i = k_i(\hmJ^s)$, for $i \leq m$.  Analogously to \eqref{mean_est}, provided $\beta\Delta$ is small (so $|I_i'|$ is large) we have
\begin{equation} \label{zetabound}
  E^X [L_{I_i'} \mid x_{b_{i-1}}=0, \mY_{[b_{i-1},a_i^s),R} ] \leq \zeta(|I_i'|),
  \end{equation}
where
\[
  \zeta(s) = \frac{K_{17}s^{c-1}}{\varphi(s)}.
  \]
For $i \geq 2$ and $t_i \in I_i'$ we have 
\begin{equation} \label{zetabound2}
  \frac{\zeta(|I_i'|)}{|I_i'|} \leq \frac{2\zeta(\ep_4 (t_i \wedge R))}{\ep_4 (t_i \wedge R)},
  \end{equation}
so bounding the maximum by twice the average we obtain
\begin{align} \label{probbound}
  P^X( \mW^s(\hmJ^s) ) &\leq \prod_{i=1}^m \sum_{t_i \in I'_i} 
    P^X(\,x_{b_{i-1}+t_i}=0\mid x_{b_{i-1}}=0,\mY_{[b_{i-1},a_i^s),R})\, p_{b_i-b_{i-1}-t_i} \\
  &\leq \prod_{i=1}^m E^X [L_{I_i'} \mid x_{b_{i-1}}=0, \mY_{[b_{i-1},a_i^s),R} ] 
    \left( \max_{t_i\in I'_i} p_{b_i-b_{i-1}-t_i} \right) \notag \\
  &\leq \left( \sum_{t_1 \in I_1'} p_{b_1-b_0-t_1} \right)
    \prod_{i=2}^m \left( \frac{2\zeta(|I_i'|)}{|I_i'|} \sum_{t_i\in I'_i} p_{b_i-b_{i-1}-t_i} \right) \notag \\
  &\leq \left( \sum_{t_1 \in I_1'} p_{b_1-b_0-t_1} \right)
    \prod_{i=2}^m \left( 4 \sum_{t_i\in I'_i} \frac{\zeta(\ep_4 (t_i \wedge R))}{\ep_4 (t_i \wedge R)} p_{b_i-b_{i-1}-t_i} \right). \notag
\end{align}
Note that for $i=1$, when \eqref{zetabound2} need not be valid, we have used the bound $|I_i'|$ in place of \eqref{zetabound}.  Using \eqref{probbound} we then bound the sum in \eqref{ZNT2} by
\begin{align} \label{tiform}
&4K_{10} \sum_{m=0}^\infty \sum_{(b_i)_{i\leq m}} \sum_{(t_i)_{i\leq m+1}} 
  (16K_{10})^m \left( \prod_{i=1}^{m+1} 
  e^{-K_{16}t_i/2R} \right) \left( \prod_{i=2}^m \frac{\zeta(\ep_4 (t_i \wedge R))}{\ep_4 (t_i \wedge R)} \right)
  \left( \prod_{i=1}^m p_{b_i - b_{i-1} - t_i} \right).  
  \end{align}
Here $0=b_0 < b_1 < \dots < b_m \leq N$, and for fixed $b_0,\dots,b_m$ the third sum includes those $t_i$ for which $t_{m+1} = N - b_m$, $b_i - b_{i-1} - t_i \geq (1-\ep_4)R$ for all $1 \leq i \leq m$, and $t_i \geq  (1 - \ep_4) M \geq  M/2$ for all $2 \leq i \leq m$.  The second sum includes those $(b_i)_{i \leq m}$ for which such $t_i$ exist. Now provided $\beta\Delta$ is small so that $M,R$ are large (depending on $\ep_4$) we see that \eqref{tiform} is bounded above by
\begin{align} \label{prodsum}
  4K_{10} &+ 4K_{10} \sum_{m=1}^\infty (16K_{10})^m
    \left( \sum_{t \geq 1} e^{-K_{16}t/2R} \right)^2
    \left( \sum_{t \geq M/2} \frac{\zeta(\ep_4 (t \wedge R))}{\ep_4 (t \wedge R)} e^{-K_{16}t/2R} \right)^{m-1}
    \left( \sum_{n \geq (1-\ep_4)R} p_n \right)^m.
  \end{align}
Now
\begin{equation} \label{sum1}
  \sum_{t \geq 1} e^{-K_{16}t/2R} \leq \frac{2R}{K_{16}},
  \end{equation}
and for some $K_{18}$,
\begin{equation} \label{sum2}
  \sum_{n \geq (1-\ep_4)R} p_n \leq \frac{K_{18}\varphi(R)}{R^{c-1}},
  \end{equation}
and for some $K_{19},K_{20}K_{21}$, provided $\beta\Delta$ is small (depending on $\ep_2,\ep_4$),
\begin{align} \label{sum3}
  \sum_{t \geq M/2} \frac{\zeta(\ep_4 (t \wedge R))}{\ep_4 (t \wedge R)} e^{-K_{16}t/2R} 
    &= \sum_{M/2 \leq t \leq R} \frac{\zeta(\ep_4 t)}{\ep_4 t} e^{-K_{16}t/2R}
    + \frac{\zeta(\ep_4 R)}{\ep_4 R} \sum_{t >R} e^{-K_{16}t/2R} \\
  &\leq \frac{K_{19}}{\ep_4} \max_{M/2 \leq t \leq R} \zeta(\ep_4 t) e^{-K_{16}t/2R} +  \frac{2\zeta(\ep_4 R)}{K_{16}\ep_4} \notag \\
  &\leq  \frac{K_{20}}{\ep_4 \varphi(R)} \left( \frac{\ep_4 R}{K_{16}} \right)^{c-1}
    +  \frac{2\zeta(\ep_4 R)}{K_{16}\ep_4} \notag \\
  &\leq \frac{ K_{21}R^{c-1} }{ \ep_4^{2-c} \varphi(R) K_{16}^{(c-1) \wedge 1} }. \notag
  \end{align}
Therefore for some $C'(R)$ and $K_{22}$, \eqref{prodsum} is bounded by
\[
  4K_{10} + C'(R)\sum_{m=1}^\infty \left( \frac{K_{22}K_{10}}{\ep_4^{2-c} K_{16}^{(c-1) \wedge 1}} \right)^m,
  \]
which is finite provided $K_{16}$ is taken sufficiently large (depending on $\ep_4$.)  Taking $\ep_2$ sufficiently small (depending on $K_{16}$) ensures that Lemma \ref{LDP_Lower} can be applied with this $K_{16}$ to obtain \eqref{ZNT2}.
\end{proof}

The reason the coarse-graining scheme in Definition \ref{semiCGdefs} is different from the one in Definition \ref{CGdefs} is that we need to avoid making a choice of $\ep_4$ (specifying the fineness of the coarse-graining scheme) that depends on $\ep_2$ (which, via $\delta_2$, determines sparse vs. dense returns.)  Having $K_{16}$ large when Lemma \ref{LDP_Lower} is applied in the proof of Proposition \ref{ZNbound} requires taking $\ep_2$ small, while on the right side of \eqref{prodsum} we need $K_{16}$ depending on $\ep_4$.  The coarse-graining scheme in Definition \ref{semiCGdefs} avoids any circularity in the choices, allowing us to specify $\ep_4$ and then $\ep_2$ depending on $\ep_4$.

\begin{proposition}\label{a.s.decay2}
Provided $\ep_2$ is sufficiently small, and $\beta\Delta$ is sufficiently small (depending on $\ep_2$), we have
\begin{eqnarray*}
\limsup_{N\to\infty}\frac{1}{N}\log Z_N^0(\,\mT(\delta_2)\,)=0.
\end{eqnarray*}
\end{proposition}

\begin{proof}
This follows immediately by Chebyshev's inequality and the previous proposition.
\end{proof}

\begin{proof}[{\bf Proof of Theorem \ref{thm}.}] Let $\rho$ be as in Lemma \ref{TN}.  Then
\begin{eqnarray*}
\rho &\leq& P^V(T_N)\\
&\leq& P^V\left( Z_N^0(\mD(\delta_2))< Z_{N,\lambda}^0 \right)\\
&\leq& P^V\left( e^{\beta f_q(\beta,\Delta)N/2}< Z_{N,\lambda}^0 \right)+ 
P^V\left( Z_N^0(\mD(\delta_2))<e^{\beta f_q(\beta,\Delta)N/2} \right)\\
&\leq& P^V\left( e^{\beta f_q(\beta,\Delta)N/2}< Z_{N,\lambda}^0 \right)+ 
P^V\left( Z_N^0 < 2e^{\beta f_q(\beta,\Delta)N/2} \right)\\
&&\qquad \qquad + P^V\left( Z_N^0(\mT(\delta_2)) > e^{\beta f_q(\beta,\Delta)N/2} \right).
\end{eqnarray*}
If $f_q(\beta,\Delta)>0$ then as $N$ tends to infinity the right hand side of the above inequality tends to zero, by Proposition \ref{a.s. decay}, Proposition \ref{a.s.decay2} and the 
fact that $ \frac{1}{N}\log Z_N^0$ tends to $\beta f_q(\beta,\Delta)$, $P^V$--a.s.  This is a contradiction so $ f_q(\beta,\Delta) = 0$.
\end{proof}


\begin{thebibliography}{99}

\bibitem{Al08} Alexander, K.S. (2008).  The effect of disorder on polymer depinning transitions.  \emph{Commun. Math. Phys.} \textbf{279} 117-146.

\bibitem{AS06} Alexander, K. S. and Sidoravicius, V. (2006).  Pinning of polymers
and interfaces by random potentials. \emph{Ann. Appl. Probab.} \textbf{16} 636--669.

\bibitem{BG04} Bodineau, T. and Giacomin, G. (2004).  On the localization transition of random copolymers near selective interfaces. \emph{J. Stat. Phys.} \textbf{117} 801--818.

\bibitem{DGLT07} Derrida, B., Giacomin, G., Lacoin, H. and Toninelli, F. L. (2007).  Fractional moment bounds and disorder relevance for pinning models.  arXiv:  math.PR/0712.2515.

\bibitem{DHV92} Derrida, B., Hakim, V. and Vannimenus, J. (1992).  Effect of
disorder on two-dimensional wetting. \emph{J. Stat. Phys.} \textbf{66}
1189--1213.

\bibitem{FLNO88} Forgacs, G., Luck, J.M., Nieuwenhuizen, Th. M. and Orland, H. (1988).
Exact critical behavior of two-dimensional wetting problems with quenched
disorder. \emph{J. Stat. Phys.} \textbf{51} 29--56.

\bibitem{Gi06} Giacomin, G. (2007). \emph{Random Polymer Models}.  Imperial College Press, London.

\bibitem{GT06a} Giacomin, G. and Toninelli, F. L. (2006).  Smoothing effect of
quenched disorder on polymer depinning transitions.  \emph{Commun. Math. Phys.} \textbf{266} (2006) 1-16.

\bibitem{NN01} Naidenov, A. and Nechaev, S. (2001). Adsorption of a random
heteropolymer at a potential well revisited:  location of transition point and
design of sequences. \emph{J. Phys. A: Math. Gen.} \textbf{34} 5625--5634.

\bibitem{Se76} Seneta, E. (1976).  \emph{Regularly Varying Functions.  Lecture Notes in Math.} \textbf{508}.  Springer-Verlag, Berlin.

\bibitem{To07} Toninelli, F. L. (2007).  Disordered pinning models and copolymers: beyond annealed bounds. \emph{Ann. Appl. Probab.}, to appear.  arXiv:0709.1629v1 [math.PR]

\bibitem{To08} Toninelli, F.L. (2008).  A replica-coupling approach to disordered pinning models.  \emph{Commun. Math. Phys.} \textbf{280}, 389-401.

\end{thebibliography}
\end{document}